%% file: Surface_Diffusion_arXiv.tex
\documentclass[11pt]{amsart}
\usepackage{amssymb, amsfonts, stmaryrd,fancyhdr,amsmath,amsthm,amsbsy,amsfonts, titletoc} 
\usepackage{latexsym,amscd, wrapfig}
\usepackage{amsmath, amsthm, multirow}
\usepackage[foot]{amsaddr}
\usepackage{amssymb}
\usepackage{comment}
\usepackage{graphicx,rotating}   %
\usepackage[utf8]{inputenc}
\usepackage{multirow,bigstrut}
\usepackage[usenames, dvipsnames]{color}
\usepackage[citecolor=blue]{hyperref}
\hypersetup{colorlinks=true,linkcolor=blue,urlcolor=blue}
\graphicspath{{Images/}}

\theoremstyle{plain}

\newcommand{\beq}{\begin{equation}}
\newcommand{\eeq}{\end{equation}}
\newcommand{\bna}{\begin{eqnarray}}
\newcommand{\ena}{\end{eqnarray}}
\newcommand{\bea}{\begin{eqnarray*}}
\newcommand{\eea}{\end{eqnarray*}}

\newcommand{\bt}[1]{ \begin{tabular} { #1 } }
\newcommand{\et} {\end{tabular}}
\newcommand{\normmm}[1]{{\left\vert\kern-0.25ex\left\vert\kern-0.25ex\left\vert #1 
    \right\vert\kern-0.25ex\right\vert\kern-0.25ex\right\vert}}

 \newcommand{\jmp}[1]{[\![#1]\!]}
 \newcommand\smbull{%
    \raisebox{-0.25ex}{\scalebox{1.7}{$\cdot$}}%
}
\input{Cal_mbb_mr}
\def\mrH{\underline{\mathrm{H}}}
\def\eps{\varepsilon}
\def\ou{\bar \rho}
\def\of{\bar f}
\def\oR{\bar R}
\def\cF{{\mathcal F}}
\def\vtau{\tau}

\def\bpm{\begin{pmatrix}}
\def\epm{\end{pmatrix}}
\DeclareMathOperator{\Span}{span}
 
   \def\rmd{\mathrm d}
   
 \newcommand{\tcb}[1]{\textcolor{blue}{#1}}
\numberwithin{equation}{section}
\theoremstyle{plain}
\newtheorem{thm}{Theorem}
\newtheorem{lemma}[thm]{Lemma}

\newtheorem{remark}{Remark}
\newtheorem{prop}[thm]{Proposition}
\newtheorem{definition}[thm]{Definition}

\title{Gradient Flows of Interfacial Energies: Curvature Agents and Incompressibility.}
\author{Keith Promislow$^1$}
\address{$^1$Michigan State University, promislo@msu.edu}
\author{Truong Vu$^2$}
\address{$^2$University of Illinois, Chicago, tvu25@uic.edu}
\author{Brian Wetton$^3$}
\address{$^3$University of British Columbia, wetton@math.ubc.ca}
\date{\today}

\begin{document}

\begin{abstract}
We present a framework for the gradient flow of sharp-interface surface energies that couple to embedded curvature active agents. We use a penalty method to develop families of locally incompressible gradient flows that couple  interface stretching or compression to local flux of interfacial mass. We establish the convergence of the penalty method to an incompressible flow both formally for a broad family of surface energies and rigorously for a more narrow class of surface energies.  We present an analysis, including a $\Gamma$-limit, of an Allen-Cahn type model for a coupled surface agent curvature energy.  \end{abstract}
\maketitle

\section{Introduction}

Surface proteins play a large role in reshaping the endomembrane system. This comprises the nuclear membrane, the endoplasmic reticulum (ER), the Golgi apparatus, lysosomes, vacuoles, and vesicles. These membranes do not merely define the boundaries of cellular organelles rather they conduct much of the synthesis, transport, and sorting of proteins and lipids. At a simple level surface proteins can induce invaginations in cell membranes that are the precursor to budding or as adjustable length that accommodate expansion of enclosed volume, see Figure \,\ref{f:Rapo} (left).   At a more complex level, a cell can rearrange its ER by adjusting the relative density of types of morphogenic surface proteins, reshaping it from a largely planar configuration into a helical packing that optimizes density or into fenestrated structures with non-trivial homology that accommodate pass-through \cite{Shemesh-14}, see Figure\,\ref{f:Rapo} (center-right).

\begin{figure}[ht]
    \includegraphics[height=1.0in]{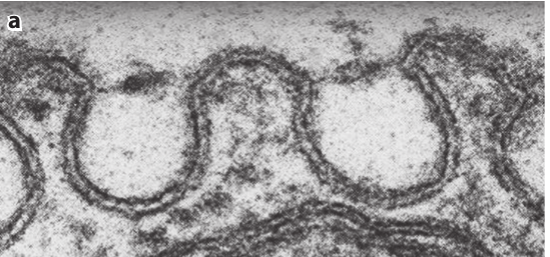}
   \includegraphics[height=1.0in]{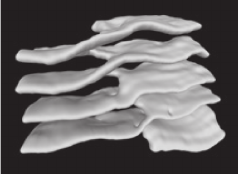}
   \includegraphics[angle=90,height=1.0in]{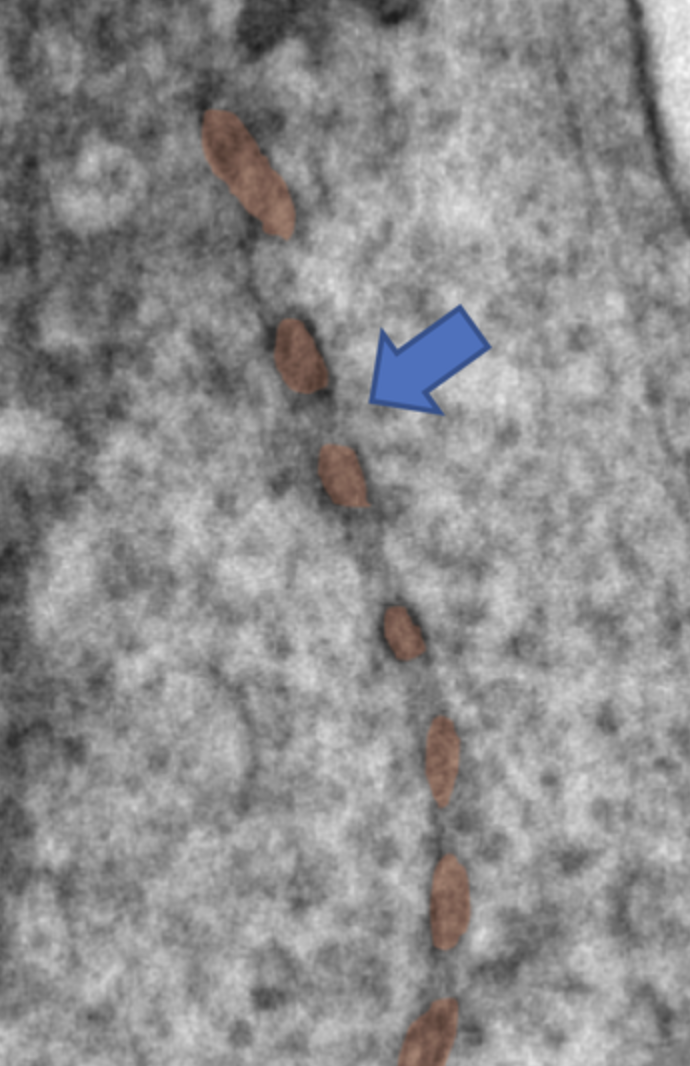}
    \caption{ (Left) Flask-like in-vaginations of caveolea in endothelial cells, curvature thought to be driven by morphogenic proteins via amphipathic helix insertion \cite{Rapo_09}, (center) reconstruction of scanning electron micrograph of helical folding of cisternea in the endoplasmic reticulum \cite{Voeltz_15}.
    (right) Optical image showing cross-sectional view of fenestrated ER. An opening within the membrane (blue arrow) \cite{Fenestated_19}.}
\label{f:Rapo}
\end{figure}

 We present a general formulation for the computation and analysis of gradients flows of  hypersurfaces induced by higher-order surface energies that couple to densities of embedded surface agents, such as surface proteins. There is a distinguished literature addressing the mean curvature and Willmore flows of smooth hypersurfaces embedded in $\mbbR^n$ in particular the seminal work of G. Huisken,
  \cite{Hui-84} and \cite{hui-86contr} for analysis of mean curvature flow. To highlight the structure of the higher-order systems we restrict attention to $n=2,$ corresponding to one-dimensional curves embedded in $\mbbR^2.$ We incorporate surface agents into the  the framework presented in \cite{NPW_25} that highlights the dichotomy afforded by the extrinsic and intrinsic representations of the surface evolution. We apply the framework to examine the role of phase separation of surface proteins in the reshaping of hypersurfaces.
  
  In section 2.5 we derive a rigorous $\Gamma$-limit for the coupled interface-surface protein energy in the limit that  the phase separation of the embedded agents becomes sharp.
  Gamma limits are well established in many cases, \cite{Federer_14}, \cite{Morgan_16}. $\Gamma$ limits of energies on fixed manifolds have been addressed in the recent literature, \cite{Uecker_23} and \cite{Uecker_25}. Our result is unique in that the surface energy couples the interface structure with the density of local agents.
   The work \cite{Roger_23} establishes compactness and lower bounds on Canham-Helfrich type energies for interfaces embedded in $\mbbR^3.$
  We incorporate coercive, higher-order derivatives of the interfacial curvature which simplifies the establishment of a $\Gamma$ limit of the energy \eqref{e:FM}. In this sense fully coupling the interface evolution to the surface density regularizes the analysis.  A recent work, \cite{Hayrap_24} addresses a similar energy in the case that the interface embedded in $\mbbR^3$ can be expressed as a graph above a flat surface in $\mbbR^2$. 
  We conclude section 2 with a characterization of the critical points of the $\Gamma-$limit energy in section 2.6.

Endomembranes are comprised of a lipid bilayer, as such they are fluidic, making them highly deformable without a reference configuration. Like many fluids they are highly incompressible. In a thin-interface limit incompressibility is manifested as a local conservation of hypersurface area. A sub-region of membrane can only grow in area if membrane material flows into it from adjacent regions. In section 3 we present a penalty method that yields a locally incompressible gradient flow for an existing surface energy. We present a formal analysis for a broad category of surface energies, showing the the gradient flow induces an effective membrane flux associated to local changes in interface length. In section 4 we extend this to a rigorous analysis for a reduced family of gradient flows and indeed to normal velocities that do not necessarily correspond to a gradient flow. The main result, Theorem\,\ref{t:ICM} presents an orbital stability result for the penalized system. We define an ``incompressible manifold'' as a graph of above the family of admissible intrinsic variables, and show that solutions of the penalized system that start sufficiently close to the invariant manifold and that remain bounded, converge on a fast time-scale into a thin neighborhood of the incompressible manifold. Moreover we establish that the dynamics of the penalized flow are the same as those of the incompressible flow to leading order.

\subsection{Intrinsic and Extrinsic Vector Fields and Coordinates}
We focus on an evolving curve $\Gamma$ parameterized by a map $\gamma:\mbbS\times[0,T]\mapsto{\mathbb R}^2$
where $\mbbS$ is the circle of unit circumference.  Denoting the unit tangent vector and outward unit normal vector of the curve at the point $\gamma(s,t)$ by $\tau(s,t)$ and $n(s,t)$ respectively, we define the arc-length metric
\beq\label {e:metric-def}
g:=|\partial_s\gamma|,
\eeq
the  surface gradient
$$ \nabla_s f:= \frac{1}{g}\partial_s f,$$
the Laplace-Beltrami operator
$$\Delta_s f :=  \frac{1}{g} \partial_s\left( \frac{1}{g}\partial_s f \right)=\nabla_s(\nabla_s f),$$
the surface measure
 $$\rmd\sigma=g\,\rmd s,$$
and the curvature
\[
\kappa:=-\nabla_s\vtau\cdot n=-|\nabla_s \vtau|.
\]
Under these definitions a circle admits a positive constant curvature. 
The following relations hold
\beq
\label{e:tau-n}
\nabla_s\gamma=\vtau, \quad \nabla_s\vtau=-\kappa n, \quad \nabla_s n=\kappa \vtau.
\eeq
The image of $\gamma,$ is denoted
$$\Gamma:=\Im(\gamma)=\{\gamma(s)\,\bigl|\, s\in\mbbS\}.$$

The intrinsic coordinates $U:=(\kappa,g)^t$ provide a representation of the curve $\gamma$ and its image up to rigid body rotation. The derivation of $\gamma$ from $U$ follows from \eqref{e:metric-def} and \eqref{e:tau-n}.
 Specifically, denote the angle of the tangent vector $\gamma_s(s)$ to the x-axis by $\theta(s)$. The unit tangent and normal vectors $\vtau$ and $n$ can be rewritten as explicit functions of $\theta$:
$$ \vtau(\theta):=\begin{pmatrix} \cos\theta\cr \sin \theta\end{pmatrix}, \quad n(\theta):=\begin{pmatrix} \sin\theta\cr -\cos \theta\end{pmatrix}.$$
 From \eqref{e:tau-n} we obtain an ordinary differential equation in $\mbbR^3$ for $\gamma$,
\beq\label{e:gamma-theta-kappa}
\begin{aligned}
 \nabla_s \theta&=\kappa,\\
 \nabla_s \gamma&=\tau(\theta),
 \end{aligned}
 \eeq
 subject to arbitrary initial data corresponding to a rigid body motion of the image $\Gamma=\Gamma(\gamma(U)).$
A curve $\gamma:\mbbS\mapsto \mbbR^2$ is $\cC^1$ closed if its end points and tangent vectors both align at the point of periodicity of $\mbbS$. More specifically defining the jump of a function $f:\mbbS\mapsto\mbbR,$
$$\jmp{f}:=f(1)-f(0),$$
then a pair $U=(\kappa,g)^t\in H_{\rm per}^k(\mbbS)$-corresponds to a curve $\gamma$ with a $H^{k+2}$ closed image $\Gamma$ if and only if
\begin{align}
  \jmp{\gamma}= \int_\mbbS \vtau(\theta)\,\rmd \sigma &=0, \label{e:C1}\\
  \jmp{\theta}=  \int_\mbbS \kappa(s)\rmd \sigma &=2\pi.\label{e:C2}
\end{align}
The boost in smoothness of $\gamma$ arises from the relation
$$ \Delta_s \gamma=-n(\theta)\kappa,$$
with the right-hand side in $H^k.$
This establishes the admissible class of curvature-arc length pairs that lead to $H^3$ closed images,
\beq \label{e:cA-def}
\cA:=\{U\in H^1(\mbbS)\, \bigl|\, \cJ(U)=0\},
\eeq
where the jump functional $\cJ:H^1(\mbbS)\mapsto \mbbR^3$ is defined in terms of the jumps 
\beq\label{e:cJ-def}
\cJ(U)=(\jmp{\gamma},\jmp{\theta}-2\pi)^t.
\eeq
This formulation  is independent of re-parameterization of the map $\gamma.$ 

The time evolution of a curve $\gamma$ can be specified through an extrinsic vector field
$\mbbV=(\mrV, \mrW)$ where
the normal and tangential velocities, $\mrV(s,t)$ and $\mrW(s,t)$, are
defined on $\mbbS\times[0,T]$.  The extrinsic vector field induces the evolution equation
\beq \label{e:gam-vel}
\gamma_t=\mrV n+\mrW \tau.
\eeq

The extrinsic vector field induces an evolution in the intrinsic vector field, corresponding to the evolution of the first and second fundamental forms of the surface. The intrinsic and extrinsic vector fields induce dual evolution equations.  These are related through a linear map from the extrinsic (Cartesian) vector field $\mbbV$ to the intrinsic vector field $\partial_t U:=(\partial_t \kappa, \partial_t g)^t$, expressed as
\beq\label{e:In-Ext_VF} \begin{pmatrix}
    \partial_t \kappa \cr \partial_t g 
\end{pmatrix} =
\cM 
\begin{pmatrix}
 \mrV \cr \mrW   
\end{pmatrix}.
\eeq
Here the linear operator $\cM=\cM(U)$ takes the form
\beq \label{e:cM_def}
\cM:= \begin{pmatrix} \mrG & \nabla_s \kappa \cr
  g\kappa & g\nabla_s
 \end{pmatrix},
\eeq
where the geometry operator $\mrG:=-\Delta_s -\kappa^2.$ This result is derived in this context in \cite{Pismen_06} for specific cases of $\mbbV$ and more generally in \cite{Barrett_20} for interfaces immersed in $\mbbR^n$. The operator $\cM$ has a three-dimensional kernel spanned by infinitesimal generators of the rigid body motions. The kernel of its adjoint $\cM^\dag$ is spanned by 
\beq\label{e:Psidag}
\Psi_1^\dag=\bpm \gamma_1\cr -\frac{1}{g}(\kappa \gamma_1 +\tau_1)\epm,\, 
\Psi_2^\dag= \bpm \gamma_2\cr -\frac{1}{g}(\kappa \gamma_2 +\tau_2)\epm,\,
\Psi_3^\dag=
\bpm 1 \cr \frac{\kappa}{g}\epm,
\eeq
see \cite{NPW_25} or \cite{Pwu_26} for details specific to this case.

In the form \eqref{e:In-Ext_VF} the evolution requires a tangential velocity to select a specific set of intrinsic coordinates. This is defined through a choice of map $\mrW=\cT(\mrV,U)$, where $\cT:\mbbR^3\mapsto\mbbR$ amounts to a choice of gauge. Typically the choice is either the co-moving frame, for which $\cT=0$, or scaled arc length, in which the arc-length $g$ is spatially constant but temporally evolving. For scaled arc-length the tangential motion satisfies
$$ \nabla_s \mrW=-\kappa\mrV +\frac{\int_\mbbS \mrV\kappa\,\rmd\sigma}{\int_\mbbS \rmd \sigma},$$
where the constant term affords periodic invertability of $\nabla_s.$ More specifically since $g=g(t)$ is a spatial constant for scaled arc-length this implies that, up to a constant
\beq\label{e:SAL}
\cT(\mrV,U)=-\int_0^s\mrV\kappa\,\mrd \sigma + \ell(s;0) \int_\mbbS \mrV\kappa\,\mrd\sigma,
\eeq
where $\ell(s)$ percentage of total   arc-length associated to the curve segment  $\gamma([0,s]).$
Unless otherwise specificed we retain a general tangential flow map and a fixed reference domain $\mbbS$. This systematizes the derivation of the gradient flow and serves to focus attention on the structural properties of the system. 

The admissible set $\cA$ does not preclude self-intersection of the image $\Gamma(U).$ This can be precluded by various additional constraints,
or,  where physically relevant, by  adding energy terms that incorporate self-repulsion, see  \cite{NPW_25} for further discussion.
In the context considered herein, the self-intersection does not impact the evolution of the curve nor of the surface agents and we ignore it to simplify presentation.

\subsection{Notation}
For functions $\{f_\eps, g_\eps\}_{\eps>0}$ that lie in a function space $X$, we write
$$ f_\eps=g_\eps+O(\eps),$$
with error in the $\|\cdot\|_X$ norm if there exists a constant $C>0$, independent of $\eps$ sufficiently small such that
$$ \|f_\eps-g_\eps\|_{X}\le C\eps.$$
If $v\in\mbbR^d$ then $v^t$ denotes its transpose.
A smoothly closed 1 dimensional surface immersed in $\mbbR^2$ we be called a curve or an interface. In some applications these represent biological membranes and will be called membranes.

\section{Gradient flows of Surface Energies coupled to Curvature Active Agents}
In many situations of interest materials adhere to  or are embedded within an interface and are carried by its flow, including tangential motion, while simultaneously influencing the interface's surface energy.  The material derivative expresses the impact of deformation of the underlying interface on the embedded agents. In particular, the derivation of a locally mass preserving gradient flow associated to a given interfacial energy requires coupling of the embedded density to the extrinsic velocity through the material derivative. 

\subsection{Material Derivative}For a general extrinsic velocity $\mbbV=(\mrV,\mrW)$ the material derivative correspond to an evolution that locally conserves density of an embedded scalar under evolution of the curve.  Specifically let $\mbbS_0\subset\mbbS$ be a labeled subset of the evolving curve. If a density contained in the image $\gamma(\mbbS_0)$ of the curve evolves under a zero material derivative, then the change of agent mass should equal the flux of agent mass out of the boundary of the subset, irrespective of the evolution of the interface.
Taking time derivative of the total density of a scalar $\rho$ yields the relation
$$
    \frac{d}{dt}\int_\mbbS \rho\,\mrd \sigma  = \int_\mbbS \left(\rho_t +\rho\frac{\partial_t g}{g}\right)\,\rmd \sigma.
$$
Substituting for the time derivative of the metric from \eqref{e:In-Ext_VF} and integrating by parts against the tangential velocity yields
$$ \begin{aligned}
  \frac{d}{dt}\int_\mbbS \rho\,\mrd \sigma 
    &= \int_\mbbS \left(\rho_t -\mrW \nabla_s\rho  +\rho\kappa \mrV \right)\,\rmd\sigma.
    \end{aligned} $$
This motivates the material derivative
\beq
\label{e:MD}
\frac{D\rho}{Dt}:= 
\rho_t+\kappa\mrV\rho -\mrW\nabla_s \rho=\left(\partial_t U+ \mbbV\cdot \bpm \kappa \cr -\nabla_s\epm\right) \rho.
\eeq
The material derivative zero-flow
\beq\label{e:MD-flow} \frac{D\rho}{Dt}=0,
\eeq
conserves $N$ globally. More significantly,  if a subset $\mbbS_0\subset \mbbS$ is labeled, then the mass of $\rho$ over the labeled set satisfies the evolution,
$$ \frac{d}{dt}\int_{\mbbS_0} \rho \rmd\sigma =
 \int_{\mbbS_0} \left(\rho_t +\rho \nabla_s \mrW +\rho\kappa \mrV \right)\,\rmd\sigma=
\int_{\mbbS_0} \frac{D\rho}{Dt} \rmd\sigma +(\mrW \rho \hat n)\Bigl|_{\partial\mbbS_0}= (\mrW \rho \hat n)\Bigl|_{\partial\mbbS_0},$$
where $\hat n$, the normal to $\partial\mbbS_0$ within $\mbbS,$  takes values in $\{\pm1\}$. The mass of $\rho$ within the evolving set $\gamma(\mbbS_0)$ is impacted only by the flux of $\rho$ out of the boundary and not by the deformation of the interface.


\subsection{Energy Gradient Formulation}
Consider a general energy which couples the intrinsic variables $U$ to a vector-valued density $\rho$
\beq\label{e:Eng_V}
\cF(U,\rho):=\int_\mbbS \mbbF(U,\rho)\,\rmd \sigma.
\eeq
Here $(U,\rho)\in\mbbR^{2+d}$ are the free variables and $\mbbF$ depends upon $U$ and $\rho$ through their surface derivatives $(\nabla_s)^k$ up to order $\bar k\in\mbbN_+.$ 
Imposing the closure conditions $\cJ(U)=0$ through the jump functional defined in \eqref{e:cJ-def},
defines the set of intrinsic variables $U$  that generate  $H^{\bar k+2}$ closed curves $\Gamma=\Gamma(U).$  The admissible set of intrinsic coordinates
$$\cA=\{U\in H^{\bar k}(\mbbS)\,\bigl|\, \cJ(U)=0\},$$
is invariant under any intrinsic evolution generated by an extrinsic flow
$$ \partial_t U = \cM \mbbV, $$
that leaves the curve smooth,
see \cite{NPW_25} for details.  The chain rule implies
$$ \frac{d}{dt} \cF = \int_\mbbS \nabla_\mrI \cF \cdot \bpm \partial_t U\cr \rho_t\epm
\, \rmd \sigma,$$
 where the intrinsic gradient 
 \beq\label{e:Int_grad}\nabla_\mrI \cF=(\partial_\kappa \cF, \partial_g\cF, \partial_\rho\cF^t)^t,
 \eeq 
 takes values in $\mbbR^{d+2}$ and the partial derivative notation applied to $\cF$ denotes {\sl variational derivatives} of $\cF$ with respect to the indicated variable.
 The variational derivative of $\cF$ with respect to arc-length plays a significant role and its calculation involves unfolding the impact of $g$ on powers of the spatial gradient and the surface measure, see \cite{NPW_25} for details.
Converting to an extrinsic velocity through \eqref{e:In-Ext_VF} and substituting the material derivative for $\rho_t$ we obtain
$$
\begin{aligned}
 \frac{d}{dt} \cF &=
 \int_\mbbS \bpm\partial_\kappa \cF\cr \partial_g\cF \epm \cdot \cM\mbbV + \partial_\rho\cF  \cdot\left(\frac{D\rho}{Dt}- \bpm \kappa \rho & -\nabla_s \rho \epm \mbbV\right) \,\rmd \sigma,\\
 &= \int_\mbbS \left( \cM^\dag \bpm\partial_\kappa \cF\cr \partial_g\cF \epm  +\partial_\rho\cF^t  \bpm -\kappa \rho & \nabla_s \rho \epm  \right)\cdot  \mbbV  + \partial_\rho\cF\cdot \frac{D\rho}{Dt}\, \rmd \sigma. 
 \end{aligned}$$
We introduce the $(d+2)\times 2$ augmented intrinsic-extrinsic flow map
 \beq \label{e:cN-def} \cN:= \bpm \mrG  &\nabla_s \kappa \cr
  g\kappa & g\nabla_s \cr
  -\kappa \rho & - \nabla_s \rho\epm,
  \eeq
  and its $2\times(d+2)$ adjoint
  \beq\label{e:cNt-def} \cN^\dag := \bpm  \mrG & g\kappa & -\kappa \rho^t \cr
  \nabla_s \kappa &-\nabla_s(g\smbull) & \nabla_s \rho^t \epm.
  \eeq
  With this notation  the energy dissipation associated to the extrinsic flow $\mbbV$ takes the form
  \beq
  \label{e:SD-dissipation}
  \frac{d}{dt} \cF =\int_\mbbS \cN^\dag \nabla_\mrI \cF
  \cdot \mbbV + \partial_\rho \cF\cdot \frac{D\rho}{Dt}\,\rmd \sigma.
  \eeq

A key feature of the extrinsic formulation is that the energy dissipation is independent of the choice of tangential velocity.
 \begin{lemma}
 \label{l:PIP}
The system \eqref{e:SD-dissipation} satisfies the parameterization-independence property
\beq\label{e:PIP}
[\cN^t]_2\cdot \nabla_{\mrI}\cF =
 \partial_\kappa\cF\nabla_s \kappa+
 \partial_\rho\cF\nabla_s \rho
 -\nabla_s(g\partial_g\cF)=0.
 \eeq
 In particular the tangential velocity prefactor generated by  $\cN^t \nabla_\mrI \cF $ is zero.
 \end{lemma}
\begin{proof}
Since the result is infinitesimal, it is sufficient to consider an extrinsic velocity with zero normal component $\mbbV=(0,\mrW).$
The tangential velocity is equivalent to reparameterization of $\mbbS.$ In particular given $\mrW:\mbbS\times\mbbR_+\mapsto\mbbR$ define the map $m:\mbbS\times\mbbR_+\mapsto \mbbS$ via the characteristic flow
\beq\label{e:MoC}
\partial_t m(s,t)=\mrW(m(s,t),t),
\eeq
subject to $\mbbS$ periodicity. For time-independent $(U,\rho)$ define the time-dependent reparameterization flows $\tilde U:=U\circ m$ and $\tilde \rho:=\rho\circ m.$
From the method of characteristics $\frac{D\tilde\rho}{Dt}=0$. On the other hand $(U,\rho)$ and $(\tilde U,\tilde \rho)$ are related by a change of variables. In particular $\tilde\nabla_s^k \tilde U=\nabla_s^k U$, $\tilde\nabla_s^k \tilde\rho=\nabla_s^k\rho$, and $\mrd \tilde\sigma=\mrd \sigma$ where $\tilde\nabla_s=\frac{1}{\tilde g} \partial_s$ and $\tilde\sigma=\tilde g \,\mrd s.$ The energy is invariant under the change of variables induced by $m$, that is $\cF(U,\rho)=\cF(\tilde U,\tilde \rho).$ 
The result follows from \eqref{e:SD-dissipation} and the zero material derivative of $\tilde \rho.$ 
\end{proof}

For any choice of $L^2(\mbbS)$ self-adjoint linear operator $\cG\geq0$ we associated the   $L^2(\mbbS)$-$\cG$ gradient flow system  induced by $\cF$,
\beq\label{e:SD-gradflow}
\begin{aligned}
 \frac{D\rho}{Dt} & = -\cG \,\partial_\rho \cF, \\
  \mrV_\cF&:= - [\cN^t]_1\, \nabla_{\mrI} \cF= -\mrG\partial_\kappa\cF-g\kappa \partial_g \cF +\kappa \rho\partial_\rho\cF,
\end{aligned}
\eeq
where $\mrV\cF$ is the gradient flow normal velocity. The extrinsic velocity can be supplemented with an 
arbitrary tangential velocity relation $\mrW=\cT(\mrV,U)$ without impacting the dissipation mechanism. For the resulting extrinsic velocity,  $\mbbV(U)=(\mrV_\cF(U),\cT(\mrV_\cF,U))^t$,  the energy dissipation mechanism takes the form
\beq
\label{e:Diss_mech}
\frac{d}{dt}\cF(U,\rho) = -\int_\mbbS |\mrV_\cF |^2 + \left|\cG^{\frac12}\partial_\rho \cF\right|^2\,\rmd \sigma=-\int_\mbbS |\mrV_\cF|^2 + \left|\cG^{-\frac12}\frac{D\rho}{Dt}\right|^2\,\rmd \sigma,
\eeq
independent of the choice of $\cT.$ Moreover, if $\cG$ annihilates the constant function, that is $\cG 1=0$, then the flow conserves the total mass of $\rho.$

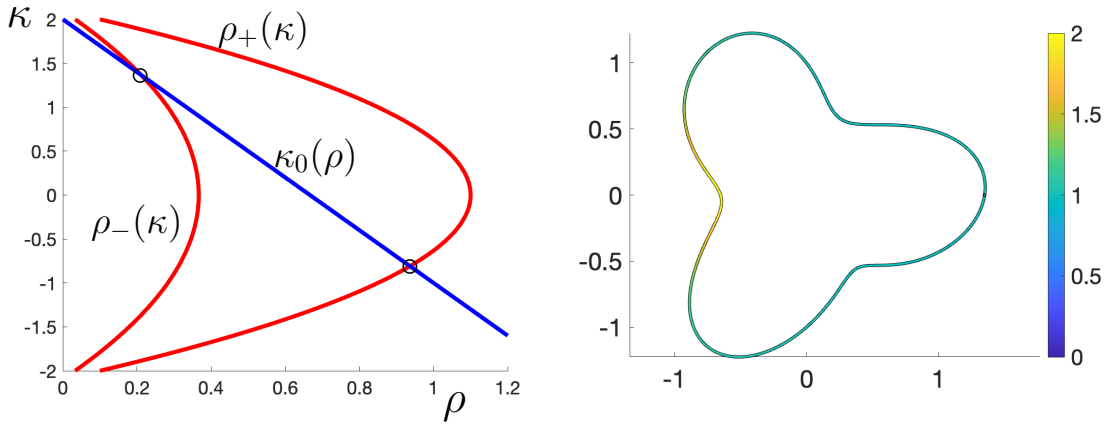
\begin{figure}[ht]
\includegraphics[height=2.25in]{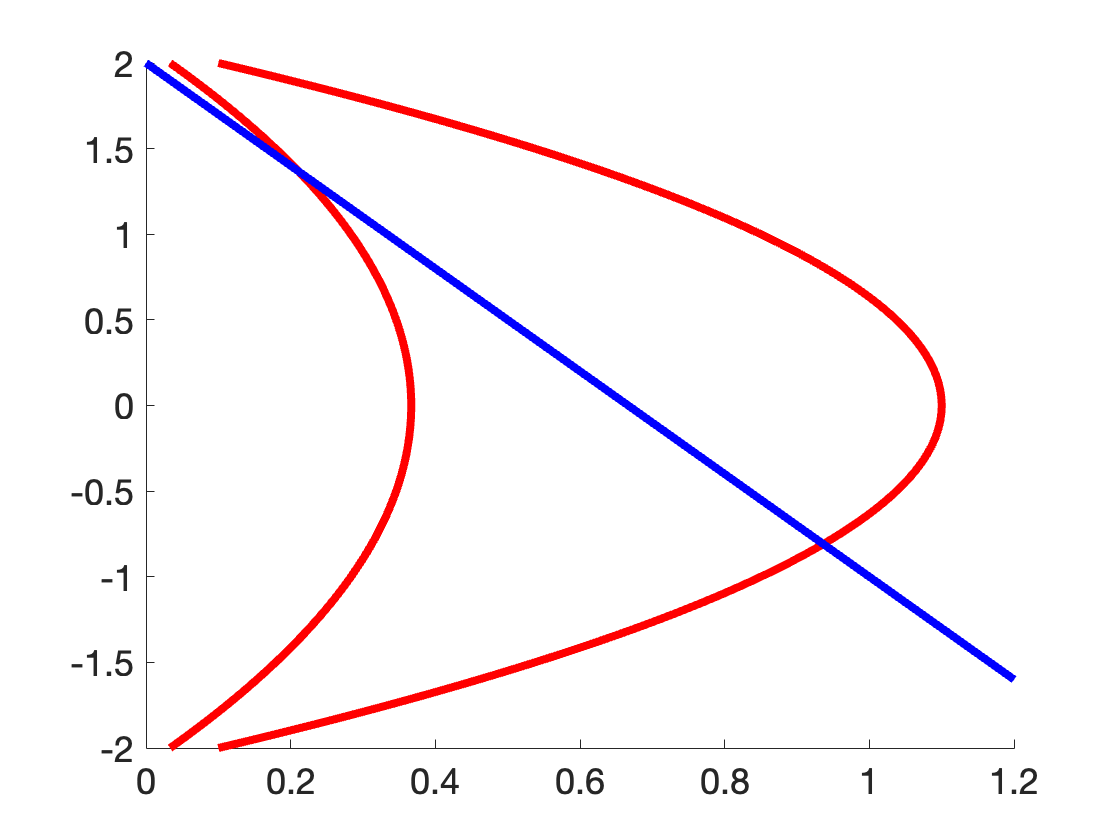}
\put(-210,148){\LARGE $\kappa$}
\put(-45,2){\LARGE $\rho$}
\put(-109,95){\Large $\kappa_0(\rho)$}
\put(-178, 70){\Large $\rho_-(\kappa)$}
\put(-130, 143){\Large $\rho_+(\kappa)$}
\put(-163,125.8){\Large $\circ$}
\put(-61, 53.5){\Large$\circ$}
\includegraphics[height=2.25in]{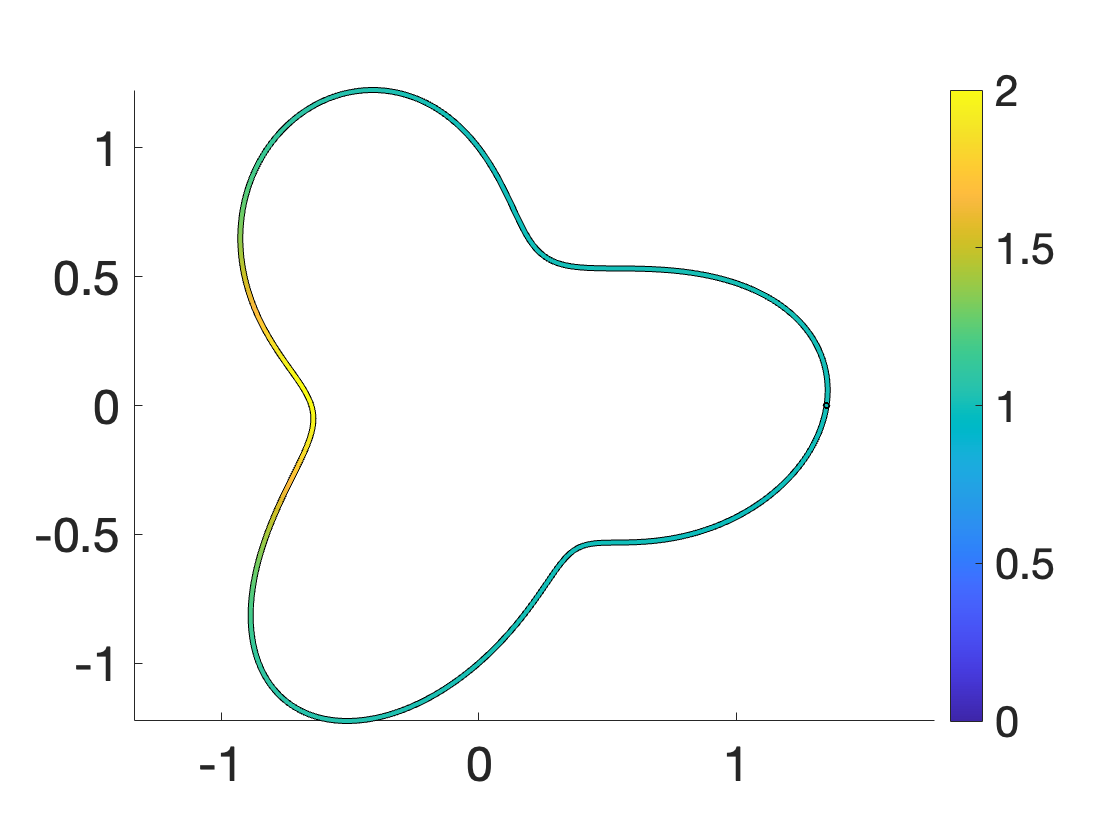}
\caption{\small (Left)
Nodal curves of $f$ (blue) and $F$ (red) showing the two double-nodal points at $(\rho_1,\kappa_1)=(0.19,1.3)$ and $(\rho_2,\kappa_2)=(0.95,-0.7)$.
 (Right) Initial curve associated to intrinsic variables $U_0$  with initial density $\rho_0$ indicated by color-bar.}
\vspace{-0.1in}
\label{f:Feps_init}
\end{figure}

\subsection{Phase Separation in Surface Proteins}
An application of the gradient flow formalism we consider a free energy for a blend of a phase-separating surface proteins that modulate a Canham-Helrich type intrinsic curvature energy. Since the potential energy is non-convex, wellposedness of the gradient flow requires  regularized by surface diffusion of both curvature and a scalar surface protein density $\rho$. We add a term term that controls total surface area   yields a free energy of the form
\beq\label{e:FM} 
\cF_\eps(U,\rho)=\int_\mbbS \Bigl(\frac{\delta}{2}|\nabla_s \kappa|^2+ \frac{1}{\delta}f(\kappa,\rho)+ \frac{\varepsilon}{2}|\nabla_s \rho|^2+ \frac{1}{\eps}\mrF(\kappa,\rho)\,\Bigr)\,\rmd\sigma+\frac{\beta}{2}(|\Gamma|-\sigma_1|\Gamma_0|)^2.
\eeq
The coefficients $\eps, \delta>0$ scale the entropic surface diffusion against enthalpy of mixing, while $\beta>0$ constrains total interface length $|\Gamma|=|\Gamma_U|$ to be proportional to the initial interface length $|\Gamma_0|$ through the parameter $\sigma_1>0.$ For $\beta\gg1$ the total membrane length is held asymptotically constant  after it reaches its quasi-steady value $\sigma_1|\Gamma_0|$.  
The potentials are defined as quadratic distances to zero nodal-line functions 
$$ \begin{aligned}
f(\rho,\kappa)&:=(\kappa-\kappa_0(\rho))^2/2,\\
\mrF(\rho,\kappa)&:=4(\rho-\rho_-(\kappa))^2(\rho-\rho_+(\kappa))^2.
\end{aligned}$$
The nodal line functions are smooth and satisfy,
$$0\leq \rho_-(\kappa)<\rho_+(\kappa)\leq1,$$
with the gap $\rho_+-\rho_-\geq \underline \mrP>0$ uniformly bounded from below by $\underline\mrP\in\mbbR_+$.
We assume that the $\kappa_0$ nodal-line graph $(\rho,\kappa_0(\rho))$ intersects each of the $\rho$ nodal-line graphs 
$(\rho_\pm(\kappa),\kappa)$ precisely once, at $(\rho_1,\kappa_1)$ and $(\rho_2,\kappa_2)$ respectively, see Figure\,\ref{f:Feps_init} (Left). We call these crossings double-nodal points.
We consider the energy over the admissible set
$U\in\cA\subset H^1(\mbbS)$ and $\rho\in H^1(\mbbS)$. We recall that $U$ is subject to the closure constraints $\cJ(U)=0,$ defined in \eqref{e:cJ-def}, while we impose a mass constraint $\rho$ is subject to a mass constraint, 
\beq\label{e:rho_MC}
\int_\mbbS \rho\,\mrd \sigma=M_\rho\in\mbbR.
\eeq

The variational derivatives with respect to $\kappa$ and $\rho$ are standard,
$$\begin{aligned}
 \partial_\rho\cF&= -\eps\Delta_s\rho +\frac{1}{\delta} f_\rho+\frac{1}{\eps}\mrF_\rho,\\
 \partial_\kappa \cF&=-\delta \Delta_s\kappa +\frac{1}{\delta} f_\kappa+\frac{1}{\eps}\mrF_\kappa.
 \end{aligned}$$
 Total arc length satisfies $|\Gamma(U)|=\int_\mbbS \mrd \sigma$ with variational derivative $\partial_g |\Gamma|= \frac{1}{g}. $ This relation and the chain rule yield
 $$
 \partial_g\cF=\frac{1}{g} \left(-\frac{\delta}{2}|\nabla_s \kappa|^2+ \frac{1}{\delta}f(\kappa,u)- \frac{\varepsilon}{2}|\nabla_s \rho|^2+ \frac{1}{\eps}\mrF(\kappa,u)\,\right)+\frac{1}{g}\beta(|\Gamma|-\sigma_1|\Gamma_0|),$$
 where the minus signs on the squared surface gradient terms arise from expanding both the surface gradients and the surface measure in $g.$ 
From \eqref{e:SD-gradflow} with choice of surface gradient $\cG=-\Delta_s,$ we determine the gradient flow
\beq\label{e:FM-GF}
\begin{aligned}
\rho_t&=\bpm \mrV\cr\mrW\epm\cdot \bpm-\kappa\rho \cr \nabla_s\rho \epm -\Delta_s\left( \eps\Delta_s\rho -\frac{f_\rho}{\delta}-\frac{\mrF_\rho}{\eps} \right),\\
\mrV_{\cF_\eps}&= \mrG\left( \delta\Delta_s\kappa-\frac{f_\kappa}{\delta}-\frac{\mrF_\kappa}{\eps} \right) -\kappa\left(-\frac{\delta}{2}|\nabla_s \kappa|^2+ \frac{f}{\delta}- \frac{\varepsilon}{2}|\nabla_s \rho|^2+ \frac{\mrF}{\eps} + \beta(|\Gamma|-\sigma_1|\Gamma_0|)\right) +\\
&\hspace{0.5in}+\kappa\rho\left(-\eps\Delta_s\rho +\frac{f_\rho}{\delta}+\frac{\mrF_\rho}{\eps}\right),
\end{aligned} 
\eeq
where  we recall that $\mrG=-(\Delta_s+\kappa^2)$ while $\rho$ and $\kappa$ subscripts on $f$ and $\mrF$ denote standard partial derivatives with respect to $\rho$ and $\kappa$ respectively.  In the numerical application the tangential velocity $\mrW=\cT(\mrV, U)$ is taken as scaled arc-length, \eqref{e:SAL}. The choice $\cG_\rho=-\Delta_s$ imposes the conservation of mass of the scalar density $\rho.$

\subsection{Numerical Simulation of Gradient Flow of $\cF_\eps$}
\label{s:PS_num}

Simulations of the system \eqref{e:FM-GF} for decreasing values of $\eps>0$ show a convergence towards a sharp phase separation limit that is investigated in Section\,\ref{s:Gamma}. The initial data is taken as depicted in Figure\,\ref{f:Feps_init} with parameter values as reported in Figure\,\ref{f:Feps}.
Since $\sigma_1=1.8>1$, there is an initial transient that induces curve lengthening, as the contribution from the $\beta$ term to $\mrV_{\cF_\eps}$ initially contributes  positive curvature term.  The flows are simulated to equilibrium using an error controlling adaptive time-stepping scheme. The equilibrium curves are roughly 80\% longer than the initial curve. This is reflected in the scale difference in initial and equilibrium curves.  
For the equilibrium configuration, the traces of $(\rho,\kappa)$ as functions of $s\in\mbbS$ are shown in Figure\,\ref{f:Feps} (right). For decreasing $\eps$ deviation from the nodal lines of $\mrF$ is increasingly energetically unfavorable, modulo jumps between the two lines. The equilibrium trace for  $\eps=0.02$ (blue diamond) is significantly closer to the red nodal line than that for $\eps=0.05$ (black *), with sharper transitions between the two nodal lines.
Since $\epsilon$ is taken smaller than $\delta$ 
the trace of $(\kappa,\rho)$ approaches the double-nodal point $(\rho_2,\kappa_2)$  ($\circ$) along the nodal line of $\mrF$.  The traces are slightly offset from the  $F$ nodal lines partially due to the impact of the constraints on the final equilibrium.  As derived in section\,\ref{s:EL} in the limit $\eps\to0$ the critical point system is 
constrained by the curve closure conditions and the $\rho$ mass conservation. 

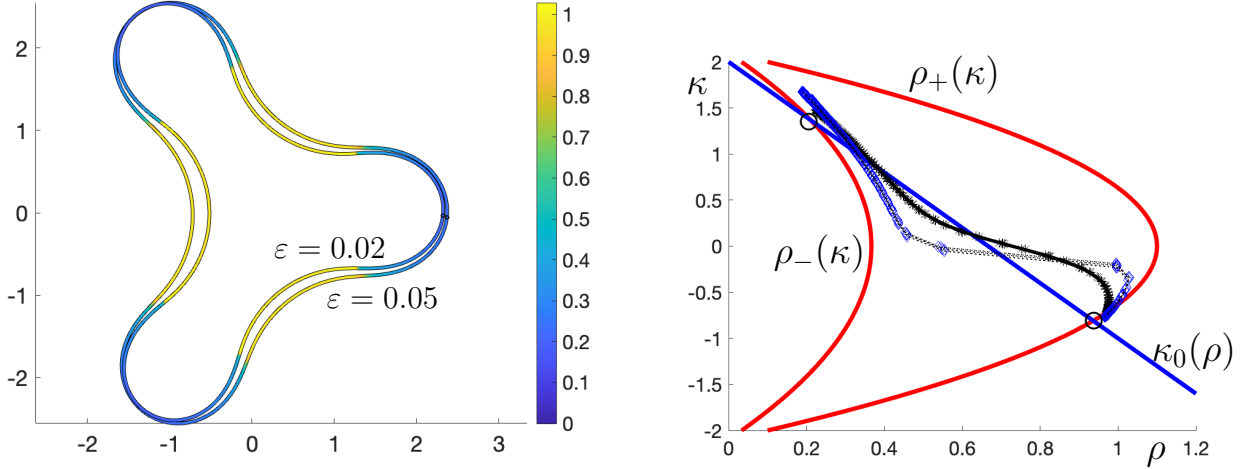
\begin{figure}[ht]
\includegraphics[width=3.6in]{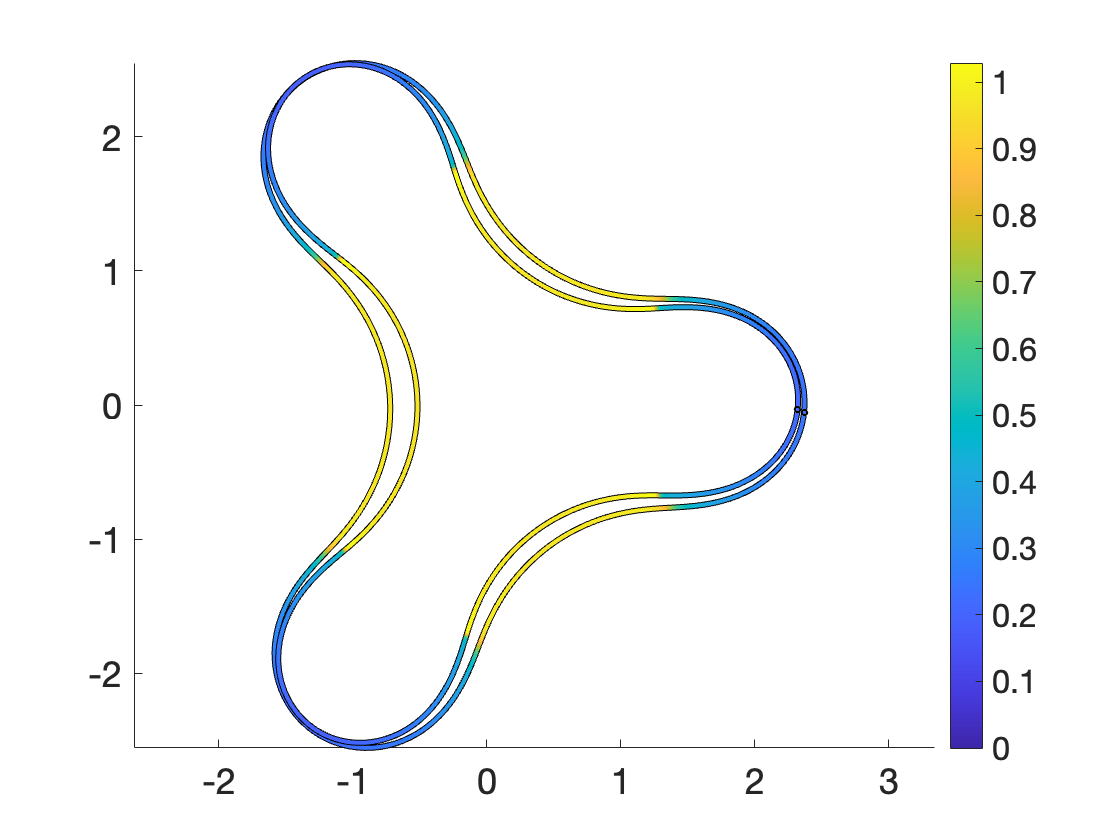}
\includegraphics[width=3.15in]{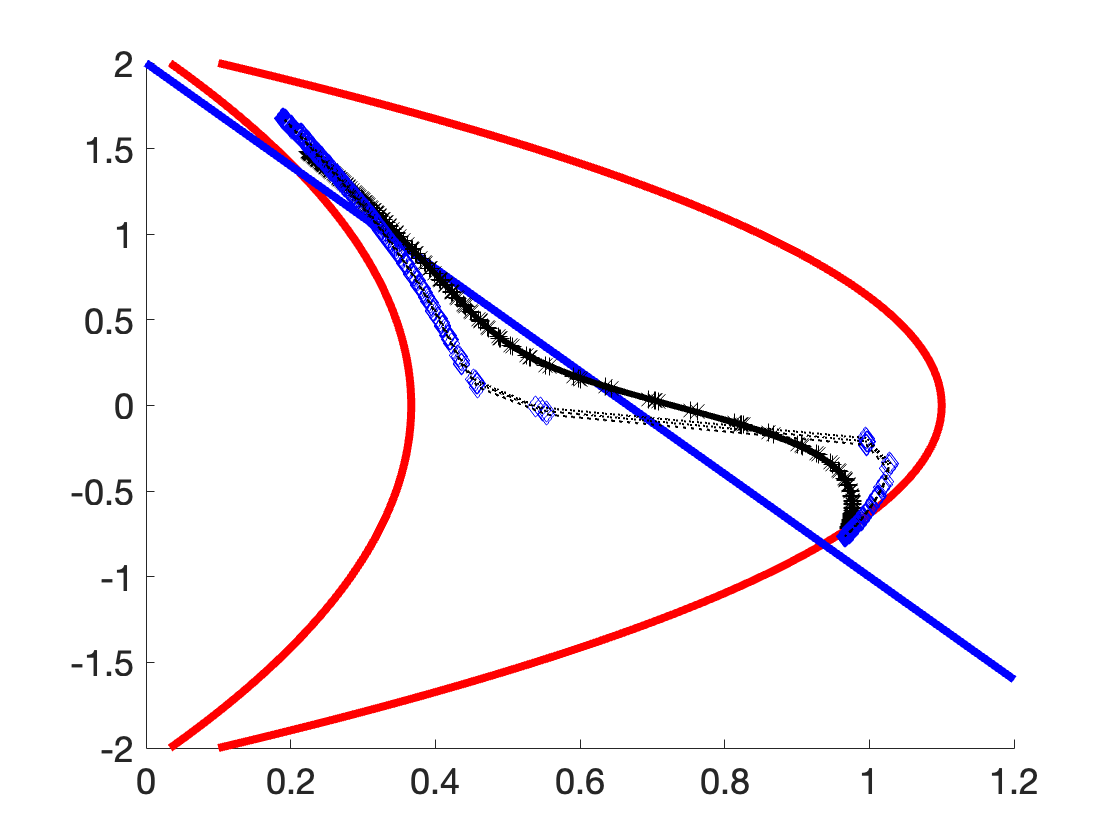}
\put(-40,8){\Large $\rho$}
\put(-214,146){\Large $\kappa$}
\put(-172,131.5){\LARGE$\circ$}
\put(-64.2,56){\LARGE $\circ$}
\put(-350,65){\large $\eps=0.05$}
\put(-370, 84){\large $\eps=0.02$}
\put(-38,45){\Large $\kappa_0(\rho)$}
\put(-181, 82){\Large $\rho_-(\kappa)$}
\put(-130, 150){\Large $\rho_+(\kappa)$}
\vspace{-0.1in}
\caption{\small (Left) Equilibrium interface  $\Gamma\subset\mbbR^2$ resulting from gradient flow with  parameters: $\delta=0.2$, $\beta=3$, $\sigma_1=1.8$ with $\eps =0.05$ and $\eps=0.02$ as labeled. Density $\rho$ indicated through color coding. (Right)
 Superposition of nodal lines of potentials with the trace of equilibrium values of $(\rho,\kappa)$  for $\eps=0.05$ (black *) and $\eps=0.02$ (blue diamond).
}
\label{f:Feps}
\end{figure}

Due to the non-monotone shape of the $\rho_\pm$ nodal lines, the maximum $\rho$ density occurs not at the double-nodal point $(\rho_2,\kappa_2)$ but along the front as it converges to double-nodal point. The density-shaded curve images of Figure\,\ref{f:Feps} (Left), show that the equilibrium resolves into three sets of alternating transition from high to low or low to high density, each with a corresponding smooth transition from positive to negative curvature. For each value of $\eps,$ these transitions induce a 6-fold covering of the trace as shown in Figure\,\ref{f:Feps} (right). This verifies the  symmetry of the equilibrium state. These numerical observations suggest an $\eps\to0^+$ $\Gamma$ limit reduction in which the density $\rho$ becomes slaved to the $\rho=\rho_\pm(\kappa)$ nodal lines, modulo fast excursions between them.

The numerical code is adapted from that described in \cite{NPW_25} and is available on the GitHub page \cite{WGHub_25}.

\subsection{Gamma Limit of $\cF_\eps$}
\label{s:Gamma}
The numerical results presented in Section\,\ref{s:PS_num} suggest that the $\eps\to0^+$ limit leads to minimizers in which the density $\rho$ follows the $\mrF$ nodal lines, except for  discontinuous transitions that provide the minimizers with flexibility to satisfy the mass and curve closure constraints. In this section we provide a rigorous analysis that establishes this as a feature of the $\eps\to0^+$  $\Gamma$-limit of $\cF_\eps.$ We first review the definition of a $\Gamma$ limit.
\begin{definition}\label{def:Gamma-conv}
Let $X$ be a complete metric space. The functional $\cF_0:X\rightarrow \mbbR\cup\{+\infty\}$ is called the $\Gamma$-limit of the functionals $\cF_\eps:X\rightarrow\mbbR\cup\{+\infty\}$ for $\eps>0$ if the following two conditions are satisfied.
\begin{itemize}
    \item The $\liminf$-inequality holds: for any sequence $\{u_\eps\}$ in $X$, we have
    \begin{equation}\label{eq:liminf-ineq}
    u =  \lim\limits_{\eps\rightarrow 0}{u_\eps} \quad  \Longrightarrow  \quad 
        \cF_0[u]\leq \liminf_{\eps\rightarrow 0}{\cF_k[u_\eps]}
    \end{equation}
    \item The $\limsup$-inequality holds: given $u\in X$ there exists a sequence $u_\eps\longrightarrow u$ in $X$ as $\eps\rightarrow 0^+$ such that
\begin{equation}\label{eq:limsup-ineq}
    \cF_0[u]\geq \limsup_{\eps\rightarrow 0}{\cF_\eps[u_\eps]}.
\end{equation}
\end{itemize}
\end{definition}

\begin{remark}
Generically there are no finite-energy limiting sequences for $\cF_\eps$ in the form \eqref{e:FM} when both $\eps$ and $\delta$ tend to zero in a distinguished limit of the form $\delta=\delta(\eps)$ with $\delta(0)=0.$  A finite energy limit would require that the limiting sequence $(U_\eps,\rho_\eps)$ converge to functions $(U,\rho)$ that only take values equal to the double-nodal points
$\{(\kappa_i,\rho_i)\}_{i=1,2}\in\mbbR^2$. With these values assumed on two complimentary sets 
$\{\mbbS_i\}_{i=1,2}$ the relation $|\mbbS_1|+|\mbbS_2|=|\mbbS|$ implies that the mass of $\kappa_\eps$ and $\rho_\eps$ is controlled by a single free parameter, for example $|\mbbS_1|$. However the constraint space prescribes the total mass of $\rho$ and  three curve closure conditions $\cJ(U)=0.$ The limiting functions will not generically satisfy these four constraints. 
A finite energy $\Gamma$-limit with $(\eps,\delta)\to (0,0)$ for $\cF_\eps$ would generically require that the functional form of the energy possesses   {\bf a minimum of five double-nodal points}.
\end{remark}

To simplify the presentation we exclude sequences that converge to a zero length or to an unbounded curve by replacing the length penalty term with a length constraint. Specifically the minimization is taken over $\cA_L\times H^1_M(\mbbS)$ where
\beq
\cA_L:= \{ U\in\cA\,\bigl|\, |\Gamma_U|=L\},
\eeq
and $H^1_M(\mbbS)$ denotes the subset of $H^1(\mbbS)$ that satisfies the mass constraint \eqref{e:rho_MC}.

The co-area formula is the essence of  the lim-inf component of the $\Gamma$-limit.  The application of the co-area formula is greatly simplified by scaling the $\kappa$ dependence from $\mrF$ 
through a change of dependent variable from $(\rho,\kappa)$ to $(\ou,\kappa)$ where,
\beq \label{e:M-def} 
\ou(\rho,\kappa)= \frac{\rho-\rho_-(\kappa)}{\mrP(\kappa)}.
\eeq
and $\mrP$ is the $\rho$ nodal gap function
\beq\label{e:Pdef}
\mrP(\kappa)=\rho_+(\kappa)-\rho_-(\kappa)\geq\underline\mrP,
\eeq
for some constant $\underline\mrP>0.$
From the inverse map
$$ \rho(\ou,\kappa)=\mrP \ou +\rho_-,$$
the $\mrF$ double well reduces to a factored form
$$ \mrF(\rho,\kappa)=\mrP^4(\kappa)\mrF_0(\ou ),$$
where the scaled $\mrF_0$ is a traditional double well potential
\beq \label{e:F0-def}
\mrF_0(\ou)=4\ou^2(\ou-1)^2.
\eeq
Introducing the rescaled functions
\begin{equation}
\begin{aligned}
    \of(\kappa, \ou) &:= f(\kappa, \mrP \ou + \rho_-), \label{e:of-def}\\
    \oR(\kappa,\ou)&:= \ou\nabla_s \mrP+\nabla_s \rho_-,
    \end{aligned}
\end{equation}
allows  the energy functional \eqref{e:FM} to be written in  a $\Gamma$-limit accessible formulation,
\begin{align}
\label{e:FM_GL}
\cF_\eps(U, \ou) &= \int_S \Big( \frac{\delta}{2} |\nabla_s \kappa|^2 + \frac{ \of(\kappa,\ou)}{\delta} 
+ \frac{\varepsilon}{2} \left|\mrP\nabla_s \ou + \oR(\kappa,\ou)\right|^2 + \frac{ \mrP^4\mrF_0(\ou )}{\eps}\Big) d\sigma.
\end{align}
To form the lim inf and lim sup sequences we introduce the sets $\mbbS_\pm$ on which the $\rho_\eps$ functions converge to the $\mrF$ nodal curves $\rho=\rho_\pm(\kappa)$ respectively. 
The transition set $\mrT=\partial\mbbS_+$
 records the domain points in $\mbbS$ where the density switches between the two $\mrF$ nodal curves.
In the unscaled density variable $\rho$ the  limiting density functions will be associated to $\mrT$ and $U$ via the map 
\beq\label{e:hatrho}\hat{\rho}(U,\mrT) = \chi_+(s)\rho_+(\kappa) + \chi_-(s) \rho_-(\kappa),
\eeq
where $\chi_\pm$ are the characteristic functions associated to $\mbbS_\pm.$ The scaled version of $\hat\rho$ is $\chi_+,$ that is $\bar{\hat\rho}=\chi_+.$
The primitive of the root-double well
\beq\label{e:h-def}
\vartheta(s):=\int_0^s \sqrt{F_0(t)}\,\rmd t,
\eeq
plays a significant role in the $\Gamma$-limit, in particular $\vartheta_1:=\vartheta(1)$ is a component of the surface tension associated to a density transition.

Our main result is the $\Gamma$-convergence of $\cF_\eps$ to
\beq\label{e:cF0-def}
\cF_0(U,\mrT)
:=  \int_{\mbbS}\frac{\delta}{2}|\nabla_s \kappa|^2+\frac{1}{\delta} f(\kappa, \hat \rho(\mrT,U))\,\rmd\sigma + \sqrt{2}\vartheta_1 \int_{\mrT}\mrP^3(\kappa)\, d\mathcal{H}^0,
\eeq
where $\cH^0$ is zero-dimensional Hausdorff measure.
With this notation we have the following result.
\begin{thm}On the admissible set $\cA_L\times H^1_M(\mbbS)$ the energy \eqref{e:FM} with $\beta=0$ satisfies
\beq
\label{e:Gamma-NE}
\Gamma\!-\!\lim\limits_{\hspace{-0.15in}\eps\to0^+}\cF_\eps =
\cF_0.
\eeq
More specifically, suppose that $(U_\eps,\rho_\eps)\subset \cA_L\times H_M^1(\mbbS)$  is an energy bounded sequence for $\cF_\eps$ as defined in \eqref{e:FM}. If $(U_\eps,\rho_\eps)\to (U,\rho)$ in $L^2\times L^1(\mbbS)$, then
  $U\in\cA_L$ and there exists a finite transition set $\mrT\subset\mbbS$  such that  $\rho=\hat\rho(U;\mrT)$. In addition we have  $$\liminf\limits_{\eps\to0^+} \cF_\eps(U_\eps,\rho_\eps)\geq \cF_0(U,\mrT).$$
\end{thm}

\begin{remark}
    The Hausdorff integral term in the $\Gamma$-limit energy \eqref{e:FM_GL} admits the explicit form
    $$\int_\mrT \mrP^3(\kappa)\,\mrd \cH^0 = \sum_{i=1}^N \left(\rho_+(\kappa(s_i))-\rho_-(\kappa(s_i))\right)^3.$$
   Modulo other constraints,  transitions are energetically favorable when they occur where the curvature $\kappa(s_i)$ enjoys a small $\mrF$ nodal gap.
\end{remark}

\begin{proof}
It is sufficient to  verify that the lim-inf and lim-sup conditions hold for the scaled energy \eqref{e:FM_GL}.
Fix $\delta$, and assume that $\bar\rho_\eps \rightarrow \bar\rho$ in $L^1(\mbbS)$ as $\eps\rightarrow 0$ and $\liminf_{\eps\rightarrow 0} \cF_\eps(U_\eps, \bar\rho_\eps)<\infty$. Since the curve length is fixed at $L$, we use scaled arc-length parameterization for which the arc length is spatially constant, e.g. $g_\eps\equiv L.$ 
Since $\delta$ is fixed we have the $\|\kappa_\eps\|_{H^1(\mbbS)}$ is uniformly bounded. By passing to a subsequence, we have $\kappa_\eps \rightharpoonup \kappa$ in $H^1(\mbbS)$, $\kappa_\eps\rightarrow \kappa$ in $L^\infty(\mbbS)$, and $\bar\rho_{\eps}\rightarrow \bar\rho$ point-wise almost everywhere as $\eps\rightarrow 0$. In particular $U\in\cA_L$ satisfies the closure constraints and the length constraint. By Fatou's lemma, we have
\beq \begin{aligned}
0&\le \int_{\mbbS}\mrP^4(\kappa)F_0(\ou)\,\mrd \sigma = \int_{\mbbS}\liminf\limits_{\eps\rightarrow 0}\mrP^4(\kappa)F_0(\ou_{\eps} )\\
&\le \liminf\limits_{\eps\rightarrow 0}\int_{\mbbS}\mrP^4(\kappa)F_0(\ou_{\eps}) \le \liminf\limits_{\eps\rightarrow 0^+}\eps \cF(U_\eps, \ou_{\eps})=0.
\end{aligned}
\eeq
Therefore, $F_0(U,\bar\rho)=0$ and $\ou$ takes the values $\{0,1\}$ almost everywhere. In particular $\bar\rho=\chi_+$ for some set $\mbbS_+$ defined through a transition set $\mrT=\partial\mbbS_+$. Moreover $\bar \mrR(\kappa_\eps,\rho_\eps)\to \bar \mrR(\kappa,\hat\rho)$ in $L^1(\mbbS)$ and $\|\bar\mrR(\kappa_\eps,\rho_\eps)\|_{L^1}$ is bounded independent of $\eps.$
This implies that
\beq
\begin{aligned} 
\frac{\varepsilon}{2} \left|\mrP\nabla_s \ou + \oR(\kappa_\eps,\ou)\right|^2 + \frac{ \mrP^4F_0(\ou )}{\eps}
&= \frac{\eps}{2}\mrP^2|\nabla_s \ou_\eps|^2 + \frac{\mrP^4}{\eps} F_0(\ou_\eps) + O(\eps),\\
&\geq \sqrt{2}\mrP^3|\nabla_s \ou_\eps |\sqrt{F_0(\ou_\eps )}+O(\eps),\\
&\geq  \sqrt{2}\mrP^{3} |\nabla_s \vartheta(\ou_\eps)| +O(\eps),
\end{aligned}
\eeq
where the error is measured in the $L^1(\mbbS)$ norm.
In addition from Fatou's Lemma the rescaled $\of$ from \eqref{e:of-def}
satisfies 
$$\liminf\limits_{\eps\to0^+}\int_\mbbS \of(\kappa_\eps,\ou_\eps)\,\mrd\sigma \geq \int_\mbbS\of(\kappa,\chi_+)\,\mrd\sigma.$$
Combining these arguments
\begin{align*}
   \liminf_{\eps\rightarrow0^+} 
   \cF_{\eps}(U_\eps,\ou_\eps)&\geq \liminf_{\eps\rightarrow0^+}\left(\int_{\mbbS} 
   \frac{\delta}{2}|\nabla_s \kappa_\eps|^2 +\frac{ \of(\kappa_\eps,\ou_\eps)}{\delta} +\sqrt{2}\mrP^3|\nabla_s \vartheta(\ou_\eps)|+O(\eps)\right)\rmd \sigma,\\
     &\ge 
 \int_{\mbbS}  \frac{\delta}{2}|\nabla_s \kappa|^2 +\frac{1}{\delta} \of(\kappa, \chi_+)\,\rmd \sigma +\sqrt{2}\liminf_{\eps\rightarrow 0^+}\left(\int_{\mbbS} \mrP^3(\kappa)|\nabla_s \vartheta(\ou_{\eps})|\,\mrd\sigma +O(\eps)\right),\\
    &\ge
     \int_{\mbbS} \frac{\delta}{2}|\nabla_s \kappa|^2 +\frac{1}{\delta} f(\kappa, \chi_+)\, \rmd\sigma+\sqrt{2}\int_{\mbbS} \mrP^3(\kappa)|\nabla_s \vartheta(\chi_+) | \rmd\sigma. 
\end{align*}
Applying the co-area formula to the last integral, and using that fact that $\vartheta(\chi_+)$  takes only the values 
$\{0,\vartheta_1\}$, allows the right-hand side to be  repackaged as the $\Gamma$-limit energy \eqref{e:Gamma-NE}. 

To establish the lim-sup condition, consider a limit point $(U,\ou)\in \cA_L\times BV(\mbbS, \{0,1\})$. Since the energy of the limiting sequence is bounded below by $\cF_0(U,\mrT)$ from \eqref{e:FM_GL}, and since the nodal gap $\mrP\geq\underline\mrP>0$ is bounded away from zero, it follows that the transition set $\mrT$ associated to a finite energy function is finite. We denote this set by  $\mrT=\{s_i\}_{i=1}^{N}\subset \mbbS$ for some $N\in\mbbZ_+$.  We assume that $\eps$ is sufficiently small that the transition gaps satisfy $s_{i+1} - s_i\gg\eps$ for each $i=1,\ldots N$ subject to the periodic identification $s_{N+1}=s_1.$ The recovery sequence is constructed by rescaling the 
$0\to 1$ heteroclinic solution to  
    \begin{align}\label{eq:equili}
        g^2\partial_s^2\varphi = \mrF_0'(\varphi),\end{align}
on the line $\mbbR.$ We translate $\varphi$ so that $\varphi(0)=\frac12.$
Fixing $\mrT$, for each $\eps>0$ sufficiently small we define the recovery sequence $\ou_{\eps}$ 
$$\ou_{\eps}(s)= \varphi\left(\frac{z \mrP(\kappa) }{\eps}\right),$$ 
where the $\mrT$-sawtooth function function $$z(s):=(-1)^{i_a}d(s,\mrT),$$ 
and $d$ is the  signed distance of $s$ to $\mrT$  achieved at argmin $s_{i_a}.$  
As depicted in Figure\,\ref{f:Sawtooth}, $z$ is piece-wise linear with a discontinuous derivative at each midpoint of any two consecutive jump locations, at which point $\phi'$ is exponentially small. The $U$ component of the recovery sequence is $U_\eps=U.$

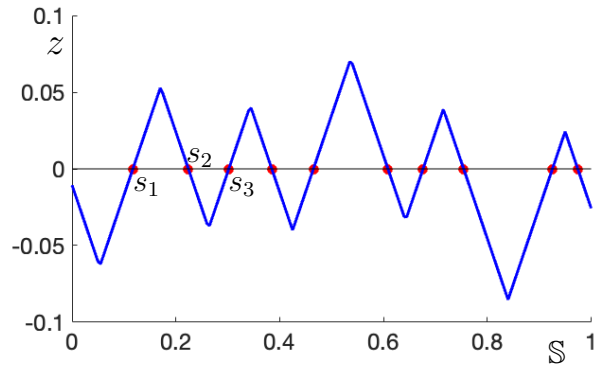
\begin{figure}[ht]
\includegraphics[width=3.5in]{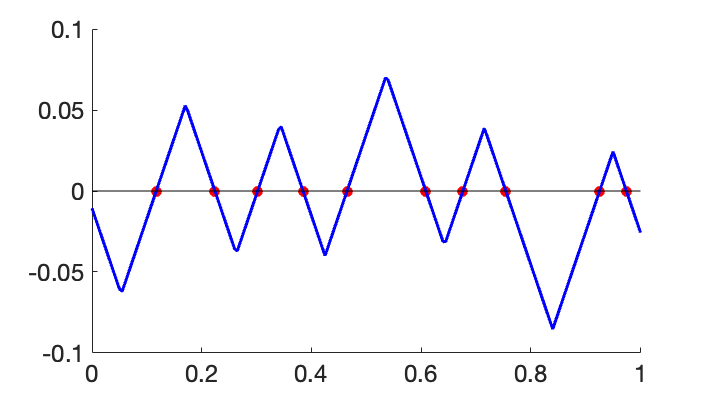}
\put(-230,117){\Large $z$}
\put(-40,-0){\large $\mbbS$}
\put(-197,65){\large $s_1$}
\put(-177,76){\large $s_2$}
\put(-161,65){\large $s_3$}

\caption{\small Plot of $\mrT$-sawtooth function $z$ for a switching set $\mrT=\{s_1,\ldots s_{10}\}$ indicated by red balls. }
\label{f:Sawtooth}
\end{figure}

 We calculate that 
    $$ |\nabla_s \ou_{\eps}|= |\varphi'|\frac{\mrP+ z\mrP'(s)}{\eps}.$$
 By assumption $\mrP$ is uniformly bounded away from zero, so the second term $$|\varphi'(z\mrP/\eps)|\frac{z}{\eps}$$ is uniformly bounded in $L^\infty$ due to the fast exponential decay of $|\varphi'(z\mrP/\eps)|.$ We deduce that
    $$ |\nabla_s \ou_\eps|= \eps^{-1}|\varphi'|\mrP + O(1),$$
    with the error measured in $L^\infty$.

Evaluating the energy functional \eqref{e:FM_GL} at the recovery sequence $\cF_\eps(U,\ou_\eps)$
attention is focused on the last two terms inside the integral. For the term
\[
\frac{\eps}{2} \int_{\mbbS} \left|\mrP\nabla \ou_{\eps} + \bar\mrR (\kappa, \ou_{\eps})\right|^2 \, \mrd\sigma,
\]
we have $\bar\mrR(\kappa, \ou_{\eps})$ is uniformly bounded, while the dominant contribution from $\nabla_s\ou_\rho$ comes from $s$ near $s_i$. 
Extracting this dominant term
\[
\frac{\eps}{2} \int_{\mbbS} \left|\mrP\nabla \ou_{\eps} + \bar\mrR (\kappa, \ou_{\eps})\right|^2 \, \mrd\sigma=
\frac{\eps}{2} \int_{B_{\eps}(s_i)} \mrP^2 \left|\nabla \ou_{\eps}\right|^2 \, d\sigma+O(\sqrt{\eps}).
\]
Changing variables $\tau(s) = \mrP d(s,\mrT)/\eps$ we have 
\beq\label{e:tau-cov}
\rmd \tau = \frac{\mrP }{\eps}\mrd \sigma +O(s\eps^{-1}).
\eeq
Keeping only leading order terms in $\eps,$ we have
\beq \label{e:rec1}
\frac{\eps}{2} \int_{B_{\eps}(s_i)} \mrP^2 \left|\nabla \ou_{\eps}\right|^2 \, d\sigma=\frac{1}{2} \mrP^3(\kappa(s_i)) \int_{-\infty}^{\infty} \left|\varphi'(z)\right|^2 \, \rmd \tau =\mrP^{3}(\kappa(s_i))  \sigma_1.
\eeq
Here the classic surface tension 
$$\sigma_1 := \frac{1}{2} \int_{-\infty}^{\infty} \left|\varphi'(\tau)\right|^2 \, d\tau=\int_{-\infty}^{\infty} F_0(\varphi(\tau)) \, d\tau,$$ 
where the second equality is the first integral of \eqref{eq:equili}. Changing variables $t=\varphi(t')$ in the definition of   $\vartheta_1$ 
 we obtain
\begin{align}
\vartheta_1 =  \int_{-\infty}^{\infty} \sqrt{F_0(\varphi(t'))} \sqrt{2 F_0(\varphi(t'))} \, \mrd t'  = \sqrt{2} \int_{-\infty}^{\infty} F_0(\varphi(t')) \, \mrd t' = \sqrt{2} \sigma_1.
\end{align}
Summing \eqref{e:rec1} over all $s_i\in\mrT$:
\[
\limsup_{\eps\to 0^+} \frac{\eps}{2} \int_{\mbbS} \left|\mrP \nabla \ou_{\eps} + \bar\mrR\right|^2 \, d\sigma 
\leq 
\sum_{i=1}^{N} \mrP^3(\kappa(s_i))  \sigma_1
= 
\sum_{i=1}^{N} \mrP^3(\kappa(s_i)) \frac{\vartheta_1}{\sqrt{2}}.
\]
For the last term in the integrand of the energy \eqref{e:FM_GL}, we consider
 $s$ is close to $s_i$,  where we have $F_0(\ou_{\eps} )\simeq F_0(\varphi(s)).$ Changing variables to $\tau$ as in \eqref{e:tau-cov} and keeping only leading order terms
yields
\beq\begin{aligned}
\label{e:rec2}
 \int_{B_{\eps}(s_i)}\!\! \frac{ \mrP^4(\kappa)F_0(\ou_{\eps} )}{\varepsilon_n}\,\mrd\sigma &=
 \mrP^3(\kappa(s_i))\int_{-\infty}^{\infty} F_0(\varphi(\tau))\,\mrd \tau
 =\mrP^3 (\kappa(s_i))\sigma_1=\mrP^3 (\kappa(s_i))\frac{\vartheta_1}{\sqrt{2}}.
\end{aligned}\eeq
Summing \eqref{e:rec2} over $s_i$, we have
\begin{align}
    \limsup_{\eps\rightarrow 0^+}   \int_{\mbbS} \frac{ \mrP^4 F_0(\ou_{\eps} )}{\eps} \,\mrd\sigma\le\sum_{i=1}^{N} \mrP^3 (\kappa(s_i))\frac{\vartheta_1}{\sqrt{2}}.
\end{align}
Adding these two contributions we deduce that
\begin{align}
      \limsup_{\eps\rightarrow0^+} \cF_{n}(U,\ou_{\eps})\leq
   \cF_0(U,\ou).
\end{align}
This completes the recovery sequence analysis.
\end{proof}

\subsection{Critical Point System for $\Gamma$-limit Energy}
\label{s:EL}
Critical points of \eqref{e:cF0-def} subject to the closure constraint \eqref{e:cJ-def} and the mass constraint \eqref{e:rho_MC} satisfy the Euler-Lagrange system.  Details on the tangent plane to the manifold of intrinsic representations of closed curves are given in \cite{NPW_25}[Prop. 3.1]. To simplify the derivation of the Euler Lagrange system we replace the constrained energy with an unconstrained Lagrangian energy associated to intrinsic coordinates $U\in\cA$ and transition set $\mrT,$
 $$
\begin{aligned}
\Lambda (U,\mrT) := \cF_0(U, \mrT) +\lambda \cdot \cJ(U)+\mu \int_\mbbS \hat\rho(\kappa,\mrT)\mrd\sigma,
\end{aligned}
$$
where $\lambda\in\mbbR^3$ and $\hat\rho$ is defined in \eqref{e:hatrho}. The admissible class is opened to $U\in H^1_L(\mbbS)$, corresponding to curves of length $L$ but which may not be closed. The Lagrange multipliers $\lambda$ are adjusted to enforce the $\cJ(U)=0$.  Moreover, for critical points of the Lagrangian energy their curvature and arc-length are generically continuous across $\mrT$ but $\nabla_s \kappa$  may be discontinuous there.  Along an admissible path $(U(\tau),\mrT(\tau))$ for $\tau\in\mbbR,$ 
$$ \begin{aligned}
    \frac{\rmd}{\mrd\tau} \cF_0 &= \int_\mbbS\left(\delta \nabla_s \kappa\nabla_s \kappa_\tau + \frac{1}{\delta} \left( \partial_\kappa f +\partial_u f \frac{\mrd\hat \rho}{\mrd \kappa}\right)\kappa_\tau + \partial_g\cF_0 g_\tau\right) \rmd \sigma+ \\ 
&\hspace{0.1in}\sum_{i=1}^N \left(\left(\frac{\delta}{2}\jmp{|\nabla_s \kappa|^2} +\mrP^3(\kappa))\right)\dot s_i +
3\vartheta_1\mrP^2(\kappa)\mrP'(\kappa)\kappa_\tau\right)\Bigl|_{s=s_i}.
\end{aligned}$$
However for $U\in H^1(\mbbS)_L$ we have $g\equiv L$ and $g_\tau=0.$ Integrating by parts we have the reduction
$$
\begin{aligned}
 \frac{\mrd}{\mrd\tau} \cF_0
&= \int_\mbbS\left(-\delta\Delta_s \kappa + \frac{1}{\delta} \left( \partial_\kappa f +\partial_\rho f \frac{\mrd\hat \rho}{\mrd \kappa}\right)\right)\kappa_\tau \,\rmd \sigma+\\
&\hspace{0.1in}\sum_{i=1}^N \left(\left(\frac{\delta}{2}\jmp{|\nabla_s \kappa|^2} +3\vartheta_1\mrP^2(\kappa)\mrP'(\kappa)\right)\dot s_i +  \left(3\vartheta_1\mrP^2(\kappa)\mrP'(\kappa)-\delta\jmp{\nabla_s\kappa} \right) \kappa_\tau\right)\Bigl|_{s=s_i},
    \end{aligned}$$
The jump function $\cJ$ and mass constraint do not involve  $\nabla_s\kappa$ so their variations are classical
$$  \frac{\mrd}{\mrd \tau}\left(\lambda \cdot\cJ(U)+\mu \int_\mbbS \hat\rho\,\mrd\sigma\right)=\sum_{i=1}^3\lambda_i\left\langle\Psi^\dag_i,\bpm \kappa_\tau\cr 0 \epm \right\rangle_{L^2(\mbbS)}\!\!\!+\lambda_4 \int_\mbbS \frac{\mrd\hat\rho}{\mrd\kappa}\kappa_\tau\,\mrd\sigma,$$
see \cite{NPW_25} and \cite{Pwu_26} for more details on the construction of variational derivative $\nabla_U\cJ$ and its connection to 
$$\ker(\cM^\dag)=\Span\{\Psi_1^\dag, \Psi_2^\dag, \Psi_3^\dag\}.$$

The contributions from $\kappa_\tau$  in the bulk integrals and $\dot{s_i}$ and $\kappa_\tau(s_i)$ in the transition point terms are independent: each must separately have a zero-coefficient to form a critical point of $\cF_0$. 
This motivations a partitioning of the domain into `free boundary' intervals 
$$\mbbS=\cup _{i=1}^N \mrI_i, $$
where $\mrI_i=[s_i, s_{i+1}]$ for $i=1,\ldots, N$. On each interval $\mrI_i$ we have an Euler-Lagrange system for $\cF_0,$
\beq\label{e:GL_EL}
\begin{aligned}
 -\delta^2 \Delta_s\kappa +\partial_\kappa f +\partial_\rho f\frac{\mrd\hat \rho}{\rmd\kappa}&=
 \delta\left(\lambda_1\gamma_1+\lambda_2\gamma_2+\lambda_3+\lambda_4 \frac{\mrd \hat\rho}{\mrd \kappa}\right).
\end{aligned}
\eeq
Here $\gamma=(\gamma_1,\gamma_2)^t$ is the curve map $\gamma:\mbbS\mapsto\mbbR^2$ induced by $U$ through \eqref{e:gamma-theta-kappa} which makes is a non-local system. At each transition point we have continuity of $\kappa$ plus two criticality conditions
\beq\begin{aligned}
\jmp{|\nabla_s \kappa|^2}_{s_i} &=-\frac{6}{\delta}\vartheta_1\mrP^2(\kappa_i)\mrP'(\kappa_i),\\
\jmp{\nabla_s\kappa}_{s_i}&=\frac{3}{\delta}\vartheta_1\mrP^2(\kappa)
\mrP'(\kappa).
\end{aligned}
\eeq
This system of scalar equations can be solved at each $s_i$, yielding nonlinear-Robin conditions
plus continuity conditions,
\beq\label{e:BCs}
\begin{aligned}
\nabla_s\kappa(s_i^+)&=-1+\frac{3\vartheta_1}{2\delta} \mrP^2(\kappa_i)\mrP'(\kappa_i),\\
\nabla_s\kappa(s_i^-)&=-1-\frac{3\vartheta_1}{2\delta} \mrP^2(\kappa_i)\mrP'(\kappa_i), \\
\jmp{\kappa}_{s_i}&=0.
\end{aligned}
\eeq
Collectively the system \eqref{e:GL_EL} subject to \eqref{e:BCs} represents  $N$ nonlocal second-order equations, subject to $3N$ boundary conditions. The equation count is balanced by considering $N$ of the boundary conditions as determining the location of the $N$ transition points $\mrT=\{s_i\}_{i=1}^N$.  The four Lagrange multipliers represent free parameters used impose the four constraints arising from the curve closure $\cJ(U)=0$ and the $\rho$-mass  \eqref{e:rho_MC} conditions.

\section{Incompressible Interface Flows}

 As an application of a coupled curvature-density energy we derive the gradient flow of an energy that penalizes variation in membrane density. In the limit of infinite penalty we show that the gradient flow converges to a normal velocity that renders the membrane incompressible while representing its  fluidic properties through a flux of membrane material. We consider a base energy $\hat\cF(U,\rho)$ that couples interface structure $U$ to embedded agents $\rho$. To this we add a scalar membrane density  $\rho_m:\mbbS\to\mbbR_+$ and an associated penalty term that pins $\rho_m$ to a reference value, $\rho_*>0$. We ignore the volumetric contribution of the embedded agents in this application.  The compressibility penalized energy takes the form
\beq \label{e:IEE}
\cF_{\eps}(U,\rho,\rho_m)=\frac{1}{\eps}\int_\mbbS \mrP(\rho_m) \,\rmd\sigma+\hat\cF(U,\rho),
\eeq
where $\hat\cF$ has the form \eqref{e:Eng_V} and $\mrP:\mbbR\mapsto\mbbR$ is strictly convex with a non-degenerate minima at $\bar\rho_m.$   For positive, self-adjoint membrane mass and embedded agent gradients $\cG_m$ and $\cG_\rho$  the gradient flow takes the form
\beq
\label{e:IEF}
\begin{aligned}
 \frac{D\rho}{Dt} & = -\cG_\rho \,\partial_\rho \hat \cF, \\
 \frac{D\rho_m}{Dt}&=-\cG_m\partial_{\rho_m}\cF_{\eps} = -\frac{1}{\eps}\cG_m \mrP'(\rho_m),\\
  \mrV_{\eps}&= \hat\mrV +\frac{\kappa}{\eps}\left( \rho_m\mrP'-\mrP\right),
\end{aligned}
\eeq
on the periodic domain $\mbbS.$ Here $\hat\mrV=\hat\mrV(U,\rho)$ denotes the normal velocity induced by $\hat\cF$.  The system is coupled through the normal velocity which appears in the convective derivatives.
The  dissipation mechanism takes the form
\beq\label{e:Dissmech} \frac{d}{dt} \cF_{\eps}(U,\rho,\rho_m) = 
- \int_\mbbS \left(\frac{1}{\eps}\left|\cG_m^\frac12 \mrP'(\rho_m)\right|^2+\left|\cG_\rho^\frac12\partial_\rho\cF_\eps\right|^2 +\left|\mrV_{\eps}\right|^2 \right)\mrd \sigma\leq 0.
\eeq
We establish, formally in a general framework, and rigorously in a simplified framework that the $\eps\to0^+$ limit yields an incompressible evolution of the membrane mediated through a Fredholm operator constructed as identity plus compact, curvature-weighted, inverse-Helmholtz operator.

\subsection{Formal Analysis of Incompressible Limit} 
\label{s:Formal}
We fix the penalty term $\mrP$ in the generic form
$$ \mrP(\rho_m)=\frac{1}{2} (\rho_m-1)^2,$$
where the equilibrium density is scaled to $\rho_*=1.$  For simplicity we choose $\cG_m=-\Delta_s$ and use the co-moving frame, in which the tangential velocity $\mrW=0.$
Writing out the material derivative the gradient flow \eqref{e:IEF} takes the form 
\beq \label{e:IEF_explicit}
\begin{aligned}
 \partial_t\rho & =-\kappa\rho\mrV_\eps 
                        -\cG_\rho \,\partial_\rho \cF+\frac{\kappa^2\rho}{\eps}\frac{1-\rho_m^2}{2}, \\
 \partial_t\rho_m&=-\kappa\rho_m\mrV_\eps 
           +\frac{1}{\eps}\left(\Delta_s \rho_m +\frac{\kappa^2(\rho_m^2-\rho_m^3)}{2}\right),\\
  \mrV_{\eps}&=  \hat \mrV-\frac{\kappa}{\eps}\frac{1-\rho_m^2}{2}.
\end{aligned}
\eeq
Assuming the variables $(U,\rho,\rho_m)$ to be  functions of the fast time  $t_1=\eps^{-1}t$,  yields a leading order dynamic that is largely uncoupled and independent of $\mrV_\eps.$ Dropping $O(\eps)$ terms in the intrinsic variable evolution yields the  system,
\beq \label{e:IEF_LO}
\begin{aligned}
 \partial_{t_1}\rho & =\rho \kappa^2 \frac{1-\rho_m^2}{2}, \\
 \partial_{t_1}\rho_m&=\Delta_s \rho_m +\frac{\kappa^2(\rho_m-\rho_m^3)}{2},\\
  \partial_{t_1}\kappa&=  -\mrG \left(\kappa\frac{1-\rho_m^2}{2}\right),\\
  \partial_{t_1}g& = -g \kappa^2 \frac{1-\rho_m^2}{2}.
\end{aligned}
\eeq
  The combination of the convective and variational terms in the membrane density system have produced an Allen-Cahn flow with strong diffusion associated to a 
  curvature-weighted double-well potential
$$ \partial_{t_1} \rho_m=\Delta_s \rho_m -\kappa^2\mrF'(\rho_m),$$
where $F(\rho_m)=\frac18 (1-\rho_m^2)^2.$ 
For initial data $\rho_m>0$ the solutions tend to  $\rho_m\equiv 1$ so long as $\kappa$ is non-zero on a set of positive mass. 
The curvature equation is well posed if $0\leq\rho_m\leq 1.$ Values of $\rho_m>1$ lead to an ill-posed curvature equation which requires a regularization by higher-order differential terms from the normal velocity, see \cite{CKP-19} and \cite{CP_23} for examples of the regularization in formal and rigorous frameworks respectively. The first and fourth equations in \eqref{e:IEF_LO} yield the identity  $\partial_{t_1}(\rho g)=0,$ which shows that the
embedded densities $\rho$ are conserved under the fast flow. The density of $\rho$ within any labeled segment $\mbbS_0\subset\mbbS$ is conserved since changes in $\rho$ are balanced by reciprocal changes in arc-length $g$.  

At quasi-equilibrium the dynamics are driven by perturbations of $\rho_m$ from equilibrium.  To resolve this structure and its impact on the full system we expand $\rho_m=1+\eps\rho_{m,1}(t_1)$ in \eqref{e:IEF_explicit}.
The $\rho_m$-$\mrV_\eps$ equations form a closed sub-system which has the the leading order dynamics
\beq \label{e:IEF_rho1_LO}
\begin{aligned} 
    \partial_{t_1} \rho_{m,1}&= -\hat \mrV_\eps \kappa +(\Delta_s-\kappa^2) \rho_{m,1}, \\
    \mrV_\eps &= \rho_{m,1}\kappa +\hat \mrV.
\end{aligned}
\eeq
This yields a fast-slow decomposition. The membrane mass correction term $\rho_{m,1}$ equilibrates on the fast time-scale while the interface evolves through the normal velocity on the native $t\sim O(1)$ time scale.  This suggests that the intrinsic variables $U$ are frozen on the fast time-scale and $\rho_{m,1}$ tends to a quasi-steady equilibrium that satisfies
\beq\label{e:rhom_quasi-stat} 
\rho_{m,1}= -\mrH^{-1} (\kappa \hat \mrV(U)),
\eeq
where the Helmholtz operator
$$\mrH(\mrU):= -\Delta_s +\kappa^2>0,$$
has a norm-bounded inverse, $\mrH^{-1}: L^2(\mbbS)\mapsto H_{\rm per}^2(\mbbS)$. 
In this quasi-steady regime the normal velocity reduces to an incompressible form $\mrV_\eps=\cI\hat\mrV$,
where the incompressibility operator
 \beq\label{e:cI-def}
\cI :=\mrI-\kappa\mrH^{-1}(\kappa\cdot\,)=\mrI-\kappa\left(-\Delta_s+\kappa^2\right)^{-1}(\kappa\cdot),
\eeq
is a curvature-weighted inverse Helmholtz relaxation.  In the limit $\eps\to0$ the extrinsic formulation relaxes to the incompressible evolution for 
$(U,\rho)$,
\beq \label{e:IC_GV}
\begin{aligned}
\frac{\hat D\rho}{\hat Dt} &= 
-\cG_\rho \,\partial_\rho \cF,\\
\mrV&= \cI\,\hat \mrV,
\end{aligned}
\eeq
where $\frac{\hat D}{\hat D t}$ is the material velocity associated to $(\cI\hat\mrV,\hat\mrW)$ and $\hat\mrW$ is a tangential velocity induced by 
$\cI\hat\mrV.$ The evolution of the energy $\hat\cF$ under the flow satisfies
$$ \frac{d}{dt} \hat\cF(U,\rho)= - \int_\mbbS \left(\left|\cG_\rho^\frac12\partial_\rho\hat \cF\right|^2 + (\cI \hat \mrV)\hat\mrV  \right)\mrd \sigma.$$
The operator $\cI=\cI(U):L^2(\mbbS)\mapsto L^2(\mbbS) $ is a self-adjoint and Fredholm, indeed it is a compact perturbation of the identity.
The system \eqref{e:IC_GV} is a gradient flow of the original energy $\hat\cF$, if and only if $\cI$ is non-negative. We establish this in the proposition below.

\begin{prop}
\label{p:cI-pos}
Suppose that $U\in\cA(\mbbS)$ is smooth and can be smoothly deformed within $\cA$ into a representation of a circular interface. Then the incompressibility operator $\cI=\cI(U):H^1_{\rm per}(\mbbS)\mapsto H^1_{\rm per}(\mbbS)$ defined in \eqref{e:cI-def} is non-negative with $\ker(\cI)=\Span\{\kappa\}.$
\end{prop}
\begin{proof}
It is easy to verify that for each $U\in\cA$,  $\kappa\in\ker(\cI(U))$. Indeed this fact implies that $\cR(\cI)\bot\kappa$ which enforces preservation of total curve length under the gradient flow.   We first consider $R\in\mbbR_+$ and define  $U_R:=(R^{-1},2\pi R)^t$ whose image $\Gamma(U)$ is a circle of radius $R$. The associated incompressibility operator takes the form
$$ \cI_R=\mrI - R^{-2}(-(2\pi R)^{-2}\partial_s^2+R^{-2})^{-1}= \mrI -(-(2\pi)^{-2}\partial_s^2+1)^{-1}.$$
In this case the kernel, $\ker(\cI_R)=\Span\{1\}$ and $\cI_R\geq \frac12$ on $\{1\}^\bot.$ For an arbitrary smooth $U$ the Helmholtz operator $\mrH=\mrH_U>0$ and hence has a bounded inverse map $\mrH^{-1}: L^2(\mbbS)\mapsto H^2_{\rm per}(\mbbS)$. The  operator $\cI(U)$ is self-adjoint Fredholm with real spectrum that accumulates at 1 and deforms continuously with smooth changes in $U$ (\tcb{cite Kato}).  If $\hat U:\mbbS\times[0,1]\mapsto\mbbR^2$ represents a smooth deformation of intrinsic variables with $\hat U(\cdot,0)=U_R$ to $\hat U(\cdot,1)=U$, then either  $\cI(\hat U(\tau))>0$ on 
$\{\kappa(\tau)\}^\bot$ or dim(ker($\cI(U(\tau')))\geq2$ for some $\tau'\in(0,\tau].$ 

For any $U\in\cA$ if $\cI u=0$ for some $u\in L^2(\mbbS)$ then 
$$ u=\kappa\mrH^{-1}(\kappa u),$$
where the operators $\mrH$ and $\mrH^{-1}$ are defined subject to periodic boundary conditions.  This implies that we can define $\frac{u}{\kappa}\in \cR(\mrH^{-1})\subset H^2_{\rm per}(\mbbS)$. Acting on this function with $\mrH$ yields the relation
$$
(-\Delta_s+\kappa^2)\frac{u}{\kappa}=\kappa u, \hspace{0.5in}\implies \hspace{0.5in}
\Delta_s \frac{u}{\kappa}=0.
$$
A function $\phi\in H^2(\mbbS)$ is in the kernel of $\Delta_s$ if it solves a second-order, linear, ordinary differential equation. There can be at most two linearly independent solutions, moreover they can be explicitly constructed. Taking the  intersection of this collection with $H^2_{\rm per}(\mbbS)$ yields
$$\ker(\Delta_s)= \Span\left\{1,\int_0^s \rmd\sigma\right\}\cap H^2_{\rm per}(\mbbS)= \Span\{1\},$$
since the second element cannot be periodic and hence cannot reside in $H^2_{\rm per}(\mbbS).$
We deduce that $\ker(\cI(U))=\Span\{\kappa\}$ with $\cI$ strictly positive on $\{\kappa\}^\bot$ so long as $U$ can be smoothly deformed from some $U_R$ on a path that resides within $\cA.$
\end{proof}

\begin{remark}
 We conjecture that all smooth $U\in\cA$ that do not self-intersect can be smoothly deformed into $U_R$ on a path that resides in $\cA$. While  self-intersect does not universally preclude smooth deformation into a circle,  a figure-$8$ curve cannot be smoothly deformed into a circle.
\end{remark}

To interpret the incompressible flow \eqref{e:IC_GV} in terms of a membrane mass flux it is convenient to introduce the limiting membrane excess density
\beq\label{e:LM_rho}
\bar\rho_m:= \lim_{\eps\to 0} \frac{\rho_m-1}{\eps}.
\eeq
The relation \eqref{e:rhom_quasi-stat} establishes a functional connection between  $\bar\rho_m$ and the corresponding $U\in\cA,$
\beq\label{e:barRM-def}
\bar\rho_m(U)= -\mrH^{-1}\kappa\hat\mrV(U).
\eeq
Let  $J_m=-\nabla_s$ be a linear operator that generates the limiting membrane diffusive mass flux 
$\mrJ_m\rho_m$.

\begin{prop}
\label{p:Flux}
Let $U$ solve the reduced flow \eqref{e:IC_GV}. For any measurable $\mbbS_0\subset\mbbS$ the change in length of the image of $\mbbS_0$ under the flow  equals the total mass flux $\mrJ_m \bar\rho_m$ out of the boundary of $\mbbS_0.$  
\end{prop}
\begin{proof}
Let $U$ solve \eqref{e:IC_GV} subject to the original gradient normal velocity $\hat \mrV$. Consider the flow of $U$ under the co-moving gauge $\mbbV=(\cI\hat\mrV,0)^t$ for which $\partial_t g=g\kappa\cI\hat\mrV.$
The evolution of the size of the image of $\mbbS_0$ satisfies
$$\begin{aligned}
\frac{d}{dt}|\gamma(\mbbS_0)| =\frac{d}{dt}\int_{\mbbS_0}  \mrd \sigma &= \int_{\mbbS_0} \partial_t g \,\mrd s = \int_{\mbbS_0}\kappa \cI\hat\mrV \,\mrd \sigma.
\end{aligned}
$$
From \eqref{e:barRM-def} the definition of $\mrH$ implies the relation
\beq\label{e:Lap2kIV}
\Delta_s \bar\rho_m= \kappa^2\bar\rho_m+\kappa\hat\mrV = \kappa(\hat\mrV -\kappa\mrH^{-1}(\kappa\hat\mrV))=\kappa\cI\hat\mrV.
\eeq


This allows us to rewrite the image length evolution as
\beq\label{e:psi_t2}
\frac{d}{dt}|\gamma(\mbbS_0)|=  \int_\mbbS \Delta_s \bar\rho_{m}\,\mrd\sigma
=\int_{\partial \mbbS_0} \mrJ_m\bar\rho_m\cdot \hat n\,\mrd \cH^0,
\eeq
where $\cH^0$ is zero-dimensional Hausdorff measure and $\hat n:\partial \mbbS_0\mapsto \{\pm 1\},$
is the exterior normal to $\mbbS_0$.
\end{proof}

\begin{remark}
The result of Proposition\,\ref{p:Flux} can be extended to more general membrane gradient operators with the factored form $$\cG_m=J_m^\dag J_m,$$
where $J_m$ is a linear operator and $J_m^\dag$ is its $L^2(\mbbS)$ adjoint. The limiting membrane diffusive mass flux is given by $\mrJ_m\rho_m$. For a weighted dissipation operator $\cG_m=-\nabla_s( w(s)\nabla_s\smbull)$ with a fixed, positive weight $w$, the the flux operator $J_m=-\sqrt{w} \nabla_s$ with adjoint $J_m^\dag=\nabla_s(\sqrt{w}\smbull).$
\end{remark}

\section{Rigorous analysis of Incompressible Willmore Type  Flows}
We present a rigorous analysis of the long-time asymptotics of solutions of the system \eqref{e:IEF_explicit} subject to constraints on the form of the system and and to assumptions of global bounds on the solutions. For brevity we consider base energies $\hat\cF$ that depend only upon the intrinsic coordinates $U$ and not on any embedded agents. Those have limited impact on the structure of the proof but complicate the statement of the assumptions on global bounds of solutions. 

The $L^2$ gradient flow associated a $\rho$-independent energy $\hat\cF$ is prescribed through the normal velocity via the map
\beq\label{e:U_GF}
\hat\mrV(U)=-\cM^\dag(U) \nabla_\mrI \hat\cF,
\eeq
 which can be supplemented with a choice of tangential velocity map $ \mrW=\cT(\mrV,U)$. We denote $\hat\mrW= \cT(\hat \mrV,U)$ and form a combined extrinsic velocity $\hat\mbbV=(\hat\mrV,\hat\mrW)^t.$
 We require that the base energy induces a normal velocity and a $U$ gradient of the normal velocity that map bounded sets to bounded sets. More specifically we assume that the combined velocity map $\hat\mbbV(U):=(\hat\mrV(U),\cT(\hat\mrV(U),U))^t$ satisfies
\beq\label{e:hatV-bnd}
\begin{aligned}
\hat\mbbV&:[H^{k}]^2\mapsto [H^{k-2}]^2,\\
\nabla_U\hat\mbbV(U)&: [H^k]^2\mapsto [H^{k-2}]^2, 
\end{aligned}
\eeq
for $k=1, 2, 3, 4, 5.$ Generically the tangential map is smoothing and does not introduce additional regularity constraints. Since the energy $\hat\cF$ is independent of $\eps$  the norm bounds are  independent of $\eps.$

In Section\,\ref{s:Incomp_Willmore}, we establish that a base energy $\hat\cF$ whose density depends only upon $U$, and not upon surface gradients of $U$ satisfies these conditions. This class  includes the quadratic energy density that generates the Willmore flow.  On the other hand, we broaden the class of flows by removing the constraint that $\hat\mrV$ arise from a gradient flow as in \eqref{e:U_GF} but preserve the assumptions \eqref{e:hatV-bnd} on its properties as a map.  The gradient flow structure is not essential to the proof once the assumptions on the global bounds on the solutions are in place. 
 
The extrinsic evolution system takes the form
\begin{align}
   \mrV&= 
   \frac{\kappa}{2\eps}(\rho_m^2-1)+\hat\mrV(U), \label{e:Full_V} \\
   \frac{D\rho_m}{Dt}&= \frac{1}{\eps}\Delta_s \rho_m,\label{e:Full_rho}
 \end{align}
 supplemented with the tangential velocity    $\mrW=\cT(\mrV,U).$
Forming the extrinsic vector field $\mbbV:=(\mrV,\mrW)^t,$  the extrinsic velocity $\mbbV=\mbbV(U,\rho_m)$ can be viewed as a function of $(U,\rho_m)$ and replaced with the intrinsic evolution 
 \beq
 \label{e:Full_U}
 \partial_t U=\cM(U) \mbbV(U,\rho_m),
 \eeq 
 coupled to \eqref{e:Full_rho}.
 
\subsection{Convergence to the Incompressible Manifold}
A natural goal is to compare the evolution of the full system  in its 
intrinsic form \eqref{e:Full_rho}-\eqref{e:Full_U}, to the reduced flow
\beq\label{e:INC_U}
\partial_t U = \cM(U) \mbbV_\cI (U),
\eeq
where the incompressible extrinsic velocity is given by
\beq
\label{e:INC_mbbV}
\mbbV_\cI(U):= \bpm \cI\hat\mrV(U) \cr
                     \cT(\cI\hat\mrV,U) \epm,
\eeq                     
through the incompressibility operator $\cI$ defined in \eqref{e:cI-def}.
However, we do not make a direct comparison of the norms of differences between solutions of the two systems, as those will diverge in time. Rather we establish an asymptotic stability result that illuminates the role of the incompressible manifold $\mrM_\rho$ as an organizing structure for the full flow.  Specifically we recall the excess membrane density $\bar\rho_m$ from \eqref{e:barRM-def} and introduce the incompressible manifold as the graph of $\bar\rho_m$ over $\cA,$
\beq\label{e:IMan}
\mrM_\rho:= \left\{ (U,1+\eps\bar\rho_m(U))\,\bigl|\, U\in \cA\right\},
\eeq
and the distance of a function $(U,\rho_m)$ to the incompressible manifold
\beq\label{e:dMPhi}
d_\mrM(U,\rho_m):= \|\rho_m-(1+\eps\bar\rho_m(U))\|_{H^2},
\eeq
the  extrinsic velocity residual associated to $(U,\rho_m)$
\beq\label{e:EV_res}
\mbbV^R :=\mbbV(U,\rho_m)-\mbbV_\cI(U),
\eeq
and the intrinsic velocity residual
\beq\label{e:IV_res}
\partial_t U^R(U,\rho_m) :=\cM(U)\mbbV^R.
\eeq

The intrinsic system is supplemented with initial data $U_0\in\cA$. 
 We address families of solutions $\{(U,\rho_m)\}_{\eps>0}$ corresponding to $\eps$ independent initial data 
 $\{U^0,\rho_m^0\}_{\eps>0}$ that are well-prepared in the sense that there exists $T>0$, independent of $\eps$, such that the corresponding solutions satisfy
 \beq\label{e:MAss}(U(\cdot;\eps),\rho_m(\cdot,\eps))^t\in L^\infty\Bigl([0,T]; [H^{5}]^2\times H^{2}\Bigr),
 \eeq
 where the $L^\infty$-temporal norm bound is independent of $\eps\in(0,\eps_0)$ for $\eps_0$ sufficiently small. Determing the class of flows and initial data with this property will depend upon the details of each system, this analysis is outside the scope of this work.
 
The following theorem shows that solutions of the full system that  satisfy \eqref{e:MAss} converge into an $\eps$ asymptotically thin region about the incompressible manifold $\mrM_\rho.$
The main issue is to control the map $F=\cM(U)\mbbV(U):U\mapsto \partial_t U$ we need that
$$ F: [H^{5}]^2\mapsto [H^{1}]^2$$
is uniformly bounded independent of $\eps$. The proof does not require $\eps$ independent bounds on $\partial_t\rho_m.$ 

\begin{thm}
\label{t:ICM}
 There exists $\eps_0,\delta, \nu, C>0$ such that for all  $\eps\in (0,\eps_0)$ the following hold. 
 Let $\{(U,\rho_m)\}_{\eps>0}$ satisfy \eqref{e:MAss} for some $T>0$ for initial data $(U^0,\rho^0_m)$ that satisfy  $d_\cM(0):=d_\cM(U_0,\rho_0)\leq \delta$.  Then the distance of the solution to the incompressible manifold \eqref{e:dMPhi} decays exponentially, satisfying the estimate
\beq\label{e:IM_B1}
d_\mrM(U,\rho_m) \leq d_\mrM(U_0,\rho_0) e^{-\nu t/\eps } +C\eps^2,
\eeq
for all $t\in[0,T].$
Once a solution satisfies $d_\mrM(U,\rho_m)\leq C\eps^2$, the  residuals are uniformly bounded
\beq\label{e:Resid_bnd}
\|\mbbV^R\|_{H^{2}}+\|\partial_t U^R\|_{L^2} \leq C \eps.
\eeq
\end{thm}

\begin{proof} 

For notational simplicity  we suppress the $\eps$ dependence of $(U,\rho).$ We  decompose the membrane density $\rho_m$ into an excess density term $\bar\rho_m=\bar\rho_m(U)$, through the relation introduced in \eqref{e:barRM-def}, plus a scaled error $w,$
\begin{align}
\label{e:rho_decomp}
\rho = 1 + \eps \bar\rho_m + \eps w.
\end{align}
The scaling of the error $w$ is convenient for grouping terms. The assumption $d_\mrM(U_0,\rho_0)<\delta/\eps$ only implies that $\|w_0\|_{H^2}\leq \delta/\eps.$
Lemma\,\ref{l:H2_coerc} establishes that $\mrH$ is uniformly $H^2$-coercive  for all $U$ generated by \eqref{e:Full_rho}-\eqref{e:Full_U}. This implies a uniform bound on the inverse map $\mrH^{-1}:L^2(\mbbS)\mapsto H^{2}(\mbbS).$
This result and the assumptions on $(U,\rho_m)$ and $\hat\mrV$ imply that $\bar\rho_m\in L^\infty\bigl([0,T]; H^5(\mbbS)\bigr)$ with an $\eps$-independent bound.

Substitute the decomposition \eqref{e:rho_decomp} for $\rho$ into the normal velocity equation \eqref{e:Full_V}, we have
\begin{align}
\label{e:V_resid}
\mrV&
=\cI\hat \mrV +\kappa w + \eps\frac{\kappa}{2} (\bar\rho_m +  w)^2.
\end{align}
Substituting the $\rho$ decomposition \eqref{e:rho_decomp} and the expression for the material derivative into the $\rho$ evolution \eqref{e:Full_rho} yields
\begin{align}\label{Drhosub}
\varepsilon \partial_t \bar\rho_m + \varepsilon \partial_t w 
= \Delta_s \bar\rho_m -\kappa\mrV +  \Delta_s w - \eps \left(\kappa\mrV(\bar\rho_m+w)-\mrW \nabla_s(\bar\rho_m+w)\right).
\end{align}
From its definition  $$\Delta_s \bar\rho_m = \kappa^2 \bar\rho_m + \kappa \hat\mrV=\kappa \cI\hat\mrV.$$ 
Using this expression for $\Delta_s\bar\rho_m$ and \eqref{e:V_resid} for $\mrV$ on the left-hand side cancels the $\cI\hat\mrV$ term and yields the evolution equation
\beq \label{e:w-error}
\partial_t w = -\frac{1}{\varepsilon} \mrH w + \mrR,
\eeq
where we introduced the residual
\beq\label{e:Res_def}
\mrR(w) :=\partial_t \bar\rho_m +\kappa\mrV(\bar\rho_m+w)-\mrW \nabla_s(\bar\rho_m+w)+\eps\frac{\kappa^2}{2}(\bar\rho_m+w)^2.
\eeq

To bound $w$ we start with the identity
\beq\label{e:TV_mist}
\begin{aligned}
\frac12 \frac{d}{dt} \|\mrH w\|_{L^2}^2 &= \langle \mrH w_t,\mrH w\rangle + \frac{1}{2}\left \langle (\mrH w)^2,\frac {\partial_t g}{g} \right\rangle+\langle \mrH_t w,\mrH w\rangle,\\
\end{aligned}
\eeq
where we introduced the time derivative of $\mrH,$
\beq\label{e:Ht-def}
    \mrH_t := \frac {\partial_t g}{g} \Delta_s +2\kappa\partial_t \kappa.
\eeq
We further introduce $\mrH^{\frac12}$, the positive square root of $\mrH,$ and take the $L^2$ inner product of the $w$ evolution equation with $\mrH^2 w,$ we find
\beq
\label{e:w_Eng_est}
\begin{aligned}
\frac12\frac{d}{dt}\|\mrH w\|_{L^2}^2 &= -\frac{1}{\eps}\langle \mrH^2 w, \mrH w\rangle +
\frac{1}{2}\left \langle (\mrH w)^2,\frac {\partial_t g}{g} \right\rangle+\langle \mrH_t w,\mrH w\rangle+\langle \mrH \mrR, \mrH w\rangle,\\
\leq& -\frac{1}{\eps} \|\mrH^\frac32 w\|_{L^2}^2+\frac12 \left|\frac{\partial_t g}{g}\right|\|\mrH w\|_{L^2}^2 +\|\mrH_t w\|_{L^2}\|\mrH w\|_{L^2}+\|\mrH^\frac12\mrR\|_{L^2} \|\mrH^\frac32 w\|_{L^2}.
\end{aligned}
\eeq

To fix the gauge we choose the tangential velocity
$\mrW=\cT(\mrV,U)$  as scaled arc length, satisfying the system
$$\nabla_s\mrW =-\kappa\mrV + \langle \kappa\mrV,1\rangle_{\mbbS}.$$
The reduced tangential velocity is defined as $\hat\mrW=\cT(\cI\mrV,U).$ 
In particular our assumptions imply that $\hat\mrW\in L^\infty([0,T];H^3(\mbbS))$.
In the scaled arc-length formulation $g=g(t)$, independent of $s\in\mbbS,$ which affords the simplification $\nabla_s=g^{-1}\partial_s$ and $\Delta_s=g^{-2}\partial_s^2.$
We record the following time derivative estimates
\begin{lemma}
\label{l:Ut_bnds}
Assuming the bounds \eqref{e:MAss} on $(U,\rho)$ we have the estimates
\beq\label{e:Ut_bdd}
\begin{aligned}
  \|\partial_t U\|_{L^2} +\left| \frac{\partial_t g}{g}\right| &\leq C\left(1+ \|w\|_{H^2} +\eps \|w\|_{H^2}^2\right),\\
\|\partial_t U\|_{H^1} & \leq C\left( 1+  \| w\|_{H^3} +\eps\|w\|_{H^3}\|w\|_{H^2}\right).
\end{aligned}
\eeq
\end{lemma}
\begin{proof}

From the form  \eqref{e:cM_def} of $\cM$ we have
$$ \|\partial_t U\|_{L^2} =\|\cM \mbbV \|_{L^2} \leq C\|\mbbV\|_{H^2}.$$
From \eqref{e:V_resid} we have 
$$\|\mrV\|_{H^2}\leq \|\cI\hat\mrV\|_{H^2}+\|w+\frac{\eps}{2}(\bar\rho_m+ w)^2\|_{H^2},$$
while
$$ \|\mrW\|_{H^1}=\|\cT(\mrV,U)\|_{H^1}\leq C\|\mrV\|_{L^2}$$
and is  easier to bound than $\|\mrV\|_{H^1}$. Since $\cI:H^2\mapsto\ H^2$ is norm bounded we  have the estimate
$$
\begin{aligned}
  \|\partial_t U\|_{L^2} &\leq C\left(\|\hat\mbbV\|_{H^2}+ \|w\|_{H^2} +\eps( 1+\|w\|_{H^2}\|w\|_{L^2}+\|w\|_{H^1}^2)\right).
  \end{aligned}
$$
  From the estimates \eqref{e:hatV-bnd} on $\hat\mrV$ and the bounds \eqref{e:MAss} on $(U,\rho)$ we obtain the $L^2$ estimate in \eqref{e:Ut_bdd}. The $H^1$ estimate on $\partial_t U$ follows from a similar argument and is omitted.
In the scaled arc-length formulation the $g$ evolution reduces to an ordinary differential equation
\begin{align}
    \frac{dg}{dt}=\int_{\mbbS}\kappa\mrV\rmd \sigma=g\int_\mbbS \kappa \mrV\rmd s.
    \end{align}
This yields the estimate
$$\left| \frac{\partial_t g}{g}\right| \leq \|\kappa \mrV\|_{L^\infty}\leq C \|\mrV\|_{H^1},$$
which has estimates that are qualitatively similar to, indeed more relaxed than, those on $\|\partial_t U\|_{L^2}.$
\end{proof}

\begin{lemma}
    \label{l:Ht}
    There exists $C>0$ such that for any solution $(U,\rho)$ that satisfies the bounds  \eqref{e:MAss},
     the operator $\mrH_t$ from \eqref{e:Ht-def} satisfies
\beq\label{e:HtL2-bound}
\begin{aligned}
\|\mrH_t u\|_{L^2}& \leq C \|u\|_{H^{2}}(1+\|w\|_{H^2}+\eps \|w\|_{H^2}^2).
\end{aligned}
\eeq
for each $t\in [0,T].$
\end{lemma}
\begin{proof}
From the form of $\mrH_t$ \eqref{e:Ht-def} and the bound \eqref{e:Ut_bdd} we estimate
$$\| (\partial_t g/g)\Delta_s u\|_{L^2} \leq C \|u\|_{H^{2}} (1+\|w\|_{H^2}+\eps\|w\|_{H^2}^2).$$
Similarly 
$$ \|\kappa\partial_t \kappa \,u\|_{L^2}\leq \|\partial_t U\|_{L^2} \|u\|_{L^2} \|\kappa\|_{L^\infty}\leq C\|u\|_{L^2} \|\partial_t U\|_{L^2}.$$
Applying  the estimates \eqref{e:Ut_bdd}
yields the bounds in \eqref{e:HtL2-bound}.
\end{proof}
 Returning to \eqref{e:w_Eng_est}, and applying the bounds in Lemma\,\ref{l:Ut_bnds} and \ref{l:Ht}, yields the estimates
 $$ \left|\frac{\partial_t g}{g}\right| \|\mrH w\|_{L^2}^2 + \|
\mrH_t w\|_{L^2}\|\mrH w\|_{L^2}\leq C(\|w\|_{H^2}^2+\|w\|_{H^2}^3+\eps \|w\|_{H^2}^4).$$
The key step in bounding the right-hand side of \eqref{e:w_Eng_est} is to control
$\partial_t\bar\rho_m$ inside of the residual $R$.
From its definition \eqref{e:barRM-def}, $\bar\rho_m$ satisfies
$$\mrH \bar\rho_m=-\kappa \hat\mrV(U),$$
and hence
$$ \mrH \partial_t \bar\rho_m +\mrH_t \bar\rho_m = -\partial_t \kappa \hat\mrV(U) +\kappa (\nabla_U\hat\mrV)\cdot \partial_t U.$$
Solving for $\partial_t \bar\rho_m$
$$ \partial_t \bar\rho_m = -\mrH^{-1}\mrH_t\,\bar\rho_m -\mrH^{-1}\left(\partial_t \kappa \hat\mrV(U)\right) + \mrH^{-1}\left(\kappa(\nabla_U\hat\mrV) \partial_t\right).$$
To bound $\|\partial_t \bar\rho_m\|_{H^1}$ we estimate
$$ \|\partial_t \bar\rho_m\|_{H^1}\leq C\left( \|\mrH^{-\frac12}\mrH_t \bar\rho_m\|_{L^2} + \|\mrH^{-\frac12}\partial_t U\|_{L^2}+ \|\mrH^{-\frac12} \kappa(\nabla_U\hat\mrV) \partial_t U\|_{L^2}\right).$$
The first term is uniformly bounded by assumptions \eqref{e:MAss} on $U.$  The second term is bounded by the $L^2$ norm of $\partial_t U$ given in \eqref{e:Ut_bdd}. The last term is the strongest. From assumption \eqref{e:hatV-bnd} on the properties of $\hat\mrV$ we have
$$ \left\|\mrH^{-\frac12}\left(\kappa \nabla_U\hat\mrV\, \partial_t U\right) \right\|_{L^2}\leq C\|\kappa \nabla_U\hat\mrV\, \partial_t U\|_{H^{-1}}\leq C\|\partial_t U\|_{H^1}.$$
Collectively the $H^1$ norm of $\partial_t \bar\rho_m$ is dominated by the $H^1$ norm of $\partial_t U,$
$$ \|\partial_t \bar\rho_m\|_{H^1}\leq C( 1+\|w\|_{H^3}+\eps \|w\|_{H^3}\|w\|_{H^2}).$$
The remaining terms in $\|\mrR\|_{H^1}$ satisfy milder estimates. Combining these bounds in \eqref{e:w_Eng_est} we obtain the energy estimate
\beq\label{e:w-est}
\begin{aligned}
\frac12\frac{d}{dt}\|\mrH w\|_{L^2}^2 &\leq -\frac{1}{\eps} \|\mrH^\frac32 w\|_{L^2}^2 + C\left(\|w\|_{H^2}^2+\|w\|_{H^2}^3 +\eps \|w\|_{H^2}^4\right) +\\
& \hspace{0.5in} C\left( 1+  \| w\|_{H^3} +\eps\|w\|_{H^3}\|w\|_{H^2}\right)\|w\|_{H^3}.
\end{aligned}
\eeq
Since the operator $\mrH^\frac k2$ induces a norm equivalent to the $H^k$-norm, we may convert the $H^k$ norms to equivalent induced norms in the positive terms. We subsequently subsume the smaller quadratic terms into the dominant negative term,  and use Young's inequality on the linear $\mrH^\frac32$ induced-norm term to obtain
\beq\label{e:w-est2}
\begin{aligned}
\frac12\frac{d}{dt}\|\mrH w\|_{L^2}^2 &\leq -\frac{1}{2\eps} \|\mrH^\frac32 w\|_{L^2}^2 +C\left(\eps+\|\mrH w\|_{L^2}^3+ 
\eps \|\mrH^\frac32 w\|_{L^2}^2(\|\mrH w\|_{L^2}+\|\mrH w\|_{L^2}^2)\right). 
\end{aligned}
\eeq
We make a continuation argument. Assume that
\beq\label{e:continuation}
\|\mrH w\|_{L^2}\leq \delta/\eps,
\eeq
for some $\delta$ independent of $\eps.$ For times $t>0$ for which this this holds, if $\delta$ is sufficiently small, we can absorb the positive $\eps$-scaled terms on the right-hand side into the negative term. Converting the negative term to a weaker $\mrH$ induced norm we again use the continuation bound, decreasing $\delta$ if need be, to subsume the cubic $\mrH$ induced-norm term into the negative quadratic $\mrH$ induced-norm term.  This yields a $\nu>0$, independent of $\eps$ for which
$$\frac{d}{dt}\|\mrH w\|_{L^2}^2 \leq -\frac{2\nu}{\eps} \|\mrH w\|_{L^2}^2 +C\eps.
$$
Since the $H^2$ norm is equivalent to the $\mrH$-induced norm, it follows immediately that the quantity $d_\mrM(t)=\eps \| w\|_{H^2}$ satisfies the bounds \eqref{e:IM_B1}. In particular the continuity estimate holds for all $t\in[0,T].$

From \eqref{e:V_resid} the extrinsic velocity residuals \eqref{e:EV_res}
admit the expansions
\beq\label{e:NT_res}
\begin{aligned}
\mrV_R &= \kappa w +\frac{\eps}{2}\kappa (\bar\rho_m+w)^2, \\
\nabla_s (\mrW_R)& = -\kappa \mrV_R+\langle\kappa \mrV_R,1\rangle.
\end{aligned}
\eeq
With the assumptions on $U$ plus the triangle inequality we have the estimates
\beq\label{e:NTres_est}
\begin{aligned}
\|\mrV_R \|_{H^2} & \leq C\left( \|w\|_{H^2}+ \eps \| w\|_{H^2}^2\right),\\
 \| \mrW_R \|_{H^2}& \leq C\left( \| w\|_{H^1}+\eps \| w\|_{H^1}^2\right).
\end{aligned}
\eeq
Once a solution has entered the invariant region  $\| w\|_{H^2} \leq C\eps,$ then the residual velocity bounds \eqref{e:Resid_bnd} follow directly.
\end{proof}

\subsection{Technical Lemmas}
This section contains definitions and
technical results needed in the proof of Theorem\,\ref{t:ICM}.
We introduce the metric-dependent norms
\beq\label{e:H^s_def}
\|u\|_{H^k}^2:= \int_\mbbS |\nabla_s^k u|^2+|u|^2\,\mrd \sigma
\eeq
for $k\in\mbbN_+.$
From the classical argument
$$ u^2(s_1)-u^2(s_2)=2\int_{s_1}^{s_2} u \nabla_su \mrd \sigma,$$
we obtain the estimate
$$ |u(s)|^2\leq \left(|\Gamma|^{-1}\|u\|_{L^2}+2\|\nabla_s u\|_{L^2}\right) \|u||_{L^2},$$
from which we infer that 
$$ \|u\|_{L^\infty} \leq \frac{C}{\sqrt{|\Gamma|}} \|u\|_{H^1},$$
and that each $H^k$ is an algebra for $k\geq 1.$

The first step is to establish the uniform coercivity of $\mrH.$
\begin{lemma}
\label{l:H_coerc}
For each $U\in\cA$ the spectrum of $\mrH(U)$ is real and non-negative. Fix $M>0$ in $\mbbR$ and form  the set $\cA_M\subset\cA$ comprised of $U$  whose image $\Gamma_U$ satisfies the length bound $|\Gamma_U|=\int_\mbbS \mrd\sigma \leq M.$  Then the spectrum of $\mrH(U)$ is uniformly bounded away from zero, that is there exists $\nu_1=\nu_1(M)$ given in \eqref{e:nu0-def} such that
\beq\label{e:H-L2coer}
\langle \mrH u,u\rangle \geq \nu_1 \|u\|_{H^1}^2,
\eeq
for all $U\in\cA_M$ and all $u\in H^1(\mbbS).$
In particular the spectrum satisfies  for all $U\in\cA_M.$ 
\end{lemma}
\begin{proof}
Since $\mrH$ is self adjoint and its bilinear form is nonnegative its spectrum is also real and non-negative. 
To establish a uniform lower  bound on its spectrum over $\cA_+$ we use the constraint $\int_\mbbS \kappa\mrd \sigma= 2\pi$. 
From Young's inequality we have
$$\begin{aligned}
2\pi =\int_\mbbS \kappa  \rmd \sigma &\leq \left(\int_\mbbS\kappa^2\mrd\sigma\right)^\frac12 \left(\int_\mbbS\mrd\sigma\right)^\frac12,\\
\end{aligned}
$$
from which we deduce the lower bound
$$
\| \kappa \|_{L^2} \geq  \frac{2\pi}{\sqrt{M}}.
$$
The bilinear form induced by $\mrH$ takes the form
$$\begin{aligned}
    \langle \mrH u, u\rangle_\mbbS&=\int_\mbbS |\nabla_s u|^2+\kappa^2u^2 \,\mrd \sigma.
\end{aligned}$$
 Define the set 
 $$S_\kappa:=\left\{s\in\mbbS\,\bigl|\,\kappa^2>\frac{2\pi^2}{M^2}\right\}.$$ 
 Since $|\mbbS|=1$ the lower bound on the $L^2$ norm of $\kappa$ implies that
\beq\label{e:Sk-lower}\int_{S_\kappa}\kappa^2\,\mrd\sigma \ge \frac{4\pi^2}{M}-\frac{2\pi^2}{M^2}\int_{\mbbS\backslash S_\kappa} \mrd\sigma \geq \frac{2\pi^2}{M}.\eeq
Consider a function $u\in H^1(\mbbS)$
with $\|u\|_{L^2}=1$ and
$$  \langle \mrH u, u\rangle_\mbbS\leq \nu,$$
for some $\nu>0.$
This implies that
$$ \int_{S_\kappa} u^2\mrd \sigma<\frac{\nu M^2}{2\pi^2}.$$
Let $u_\kappa\in\mbbR$ denote the infimum of $|u|$ over $S_\kappa.$ Using \eqref{e:Sk-lower} we have
$$ \nu\geq \int_{S_\kappa} |\nabla_s u|^2+\kappa^2 u^2\mrd \sigma \geq u_\kappa^2 \int_{S_\kappa}\kappa^2\mrd\sigma\geq  \frac{2\pi^2}{M} u_\kappa^2,$$
and hence
$$ u_\kappa\leq \sqrt{\frac{\nu M}{2\pi^2}}.$$
The set $S_u:=\{s\in\mbbS\,\bigl|\, u^2>\frac{1}{2M}\}$, is not empty. Indeed, since $\|u\|_{L^2}=1$ we have the lower bound on its size 
$$ \int_{S_u} u^2\mrd \sigma=1-\int_{\mbbS\backslash S_u}u^2\rmd \sigma \geq \frac12.$$
Fix an  arc-length function $g$. If $u$ takes values $u_\pm\in\mbbR$ with $u_+>u_-$ over a separation of $\ell$, then a standard calculus of variations argument on the quantity
$$ \int_0^\ell |\nabla_s u|^2\mrd \sigma,$$
for $u\in H^1([0,\ell])$ subject to $u(0)=u_-$ and $u(\ell)=u_+,$ shows that the minimum is achieved at
$$ u_*(s) =\frac{u_+-u_-}{\int_0^\ell\mrd\sigma}\int_0^s\mrd\sigma.$$
This minimizer has constant surface gradient
$$\nabla_s u_* =\frac{u_+-u_-}{\int_0^\ell\mrd\sigma}.$$
Consequently if any $u\in H^1(\mbbS)$ attains the values $u_\pm$ over a separation $\ell,$ we have the lower bound
\beq\label{e:grad-bnd}
\int_\mbbS |\nabla_s u|^2\rmd \sigma >\frac{(u_+-u_-)^2}{\int_0^\ell \mrd\sigma}\geq \frac{(u_+-u_-)^2}{M}.
\eeq
If $\nu< \pi^2/M$ then we have  $u_\kappa<\frac{1}{\sqrt{2M}}$ and may takes these values as our $u_\pm$. This affords a lower bound on the gradient integral and hence on $\nu$. Specifically \eqref{e:grad-bnd} yields
$$ \nu \geq \langle \mrH u, u\rangle > \frac{1}{M} \left(\frac{1}{\sqrt{2M}}-\sqrt{\frac{\nu M}{2\pi^2}}\right)^2> \frac{1}{2M^2}\left(1-\frac{M}{\pi}\sqrt{\nu}\right)^2.$$
This inequality can only hold if $\nu$ satisfies a universal lower bound 
\beq\label{e:nu0-def}\nu\geq\nu_0:= \frac{c}{M^2},
\eeq
for a constant $c.$ This  establishes the uniform $L^2$ coercivity of $\mrH$ over $\cA_M$. To obtain $H^1$ coercivity we observe
that 
$$\begin{aligned}
    \langle \mrH u, u\rangle &\geq \frac 12 \int_\mbbS |\nabla_s u|^2\mrd\sigma+ \frac12\langle\mrH u, u\rangle,\\ 
    &\geq \frac12\int_\mbbS |\nabla_s u|^2\mrd\sigma+   \frac {\nu_0}{2} \|u\|_{L^2}^2.
    \end{aligned} 
 $$
 Taking $\nu_1=\frac12 \min\{1,\nu_0\}$ 
 yields the $H^1$ coercivity.
\end{proof}
We extend the $H^1$ coercivity of $\mrH$ to $H^2$ coercivity subject to an additional assumption on $\kappa.$
\begin{lemma}
\label{l:H2_coerc}
Assume that $U\in\cA_M$ satisfies $\|\kappa\|_{H^1}\leq M.$
Then there exist constant $\nu_2>0$, depending only upon $M$ such that
$$\begin{aligned}
     \|\mrH u \|_{L^2} \geq \nu_2 \|u\|_{H^2},
\end{aligned}$$
for all $u\in H^2(\mbbS).$
\end{lemma}
\begin{proof}
From Young's inequality the $H^1$ coercivity implies the $H^1$ bound
$$ \|\mrH u\|_{L^2} \geq \nu_1 \|u\|_{H^1}.$$
On the other hand, integrating by parts yields the estimate
$$ \begin{aligned}
    \|\mrH u\|_{L^2}^2 &\geq \|\Delta_s u\|_{L^2}^2 - c\|\kappa\|_{H^1}^2\|u\|_{H^1}^2,\\
    &\geq \|\Delta_s u\|_{L^2}^2 +\|u\|_{L^2}^2 -c\frac{M^2}{\nu_1^2} \|\mrH u\|_{L^2}^2,
\end{aligned}$$
where $c$ is a numerical constant independent of $M$. We recover the $H^2$ bound for $\nu_2= c\nu_1/M.$
\end{proof}

\subsection{Application: Incompressible limit of Willmore Flows}
\label{s:Incomp_Willmore}
We verify that the structural assumptions  on $\hat\mrV$ required to apply Theorem\,\ref{t:ICM} are valid for the gradient flow generated by energies of the form
$$ \hat\cF(U)=\int_\mbbS \mrF(\kappa)\,\mrd\sigma.$$
We assume that $\mrF:\mbbR\to\mbbR$ is smooth,  bounded below, and tends to positive infinity as $|\kappa|\to\infty.$  This includes the Willmore flow generated by the function $\mrF(\kappa)=\kappa^2,$ The energy induces the normal velocity \eqref{e:SD-gradflow} 
$$ \hat\mrV= -\mrG\partial_\kappa\cF-g\kappa \partial_g \cF +\kappa \rho\partial_\rho\cF.
$$
which takes the nonlinear-functional form
$$ \hat\mrV(U)=-\mrG \mrF'(\kappa) - \kappa\mrF(\kappa),$$
for which the linearized operator is
$$ \nabla_U \hat \mrV = \bpm -\mrG\left(\mrF''\smbull\right)   -3\kappa\mrF' -\mrF  \cr
                           \frac{\smbull}{g} \Delta_s\mrF'-\nabla_s\left(\frac{\smbull}{g}\nabla_s\mrF'\right)\epm.$$
For $U$ satisfying the bounds \eqref{e:MAss}, it is straightforward to verify that $\hat\mrV$ satisfies the bounds \eqref{e:hatV-bnd}. The only loss of smoothness in applying $\hat\mrV$ to $U$ arises in the action of $\Delta_s$ on $\kappa$, which reduces smoothness by two derivatives. For $\nabla_U\hat\mrV$, acting out the differential terms yields a family of linear, second order differential operators  with coefficients that are in $H^3$ or smoother. Such operators take bounded sets in $H^k$ to bounded sets in  $H^{k-2}$ for $k=1...5$ as required. 

\begin{figure}[ht]
\includegraphics[width=3.5in]{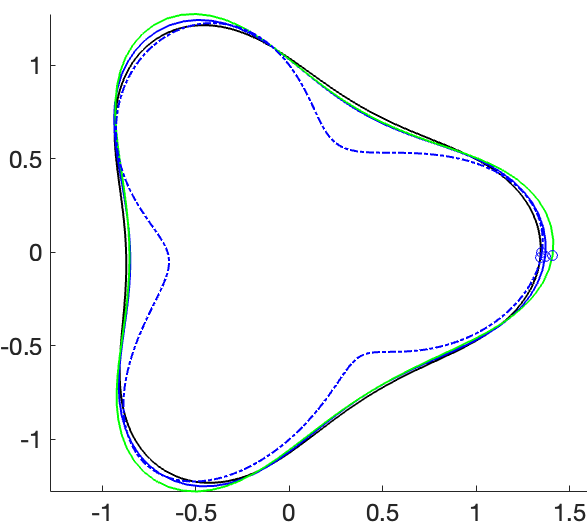}
\includegraphics[width=2.5in]{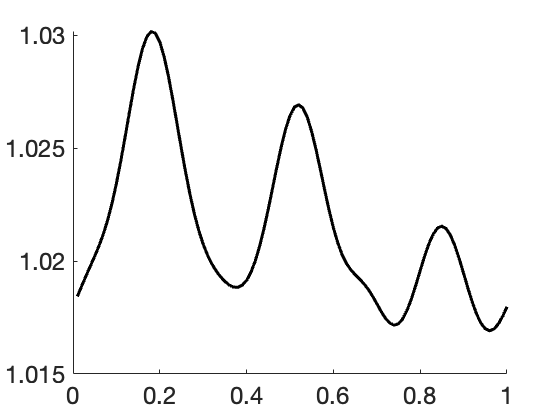}
\put(-175,135){\LARGE $\rho_m$}
\put(-30,2){\large $\mbbS$}
\caption{\small (Left) Initial curve (blue dash-dotted) and curve at time $t=0.01$ for gradient flow from the globally length preserving system \eqref{e:cF1_GF} with $\beta=5$ (black), gradient flow from local membrane density $\rho_m$ penalty \eqref{e:cF2_GF} with $\eps=0.01$ (blue), and gradient flow subject to incompressibility \eqref{e:cF2_IGF} (green). Circle on curve marks location of $\gamma(0)$. (right) Spatial distribution of membrane density  at final time for $\rho_m$ penalty gradient flow, initial $\rho_m$ was identically 1.}
\label{f:Incomp}
\end{figure}

Simulations of the gradient flow were conducted to compare three  systems that approximately or exactly preserve either total or local curve length.
The first is the gradient flow on the energy
\beq\label{e:cF1} \cF_1(U):=\int_\mbbS\frac{\kappa^2}{2}\,\mrd\sigma -\frac{\beta}{2}(|\Gamma|-|\Gamma_0|)^2.\eeq
This yields a Willmore flow with a penalty term on total curve length.
\beq \label{e:cF1_GF}
\mrV= -\mrG \kappa +\frac{\kappa^3}{2} +\kappa\beta(|\Gamma|-|\Gamma_0|).
\eeq
The second is the gradient flow subject to the local mass penalty
\beq \label{e:cF2}
\cF_2(U,\rho_m)= \int_\mbbS\frac{\kappa^2}{2}\frac{(\rho_m-1)^2}{2\eps}\,\mrd\sigma,
\eeq
and an initially uniform membrane mass $\rho_m(\cdot,0)=1.$ This yields the system
\beq\label{e:cF2_GF}
\begin{aligned}
    \mrV_\eps &= -\mrG \kappa +\frac{\kappa^3}{2} + \frac{\kappa}{2\eps} (\rho_m^2-1),\\
    \frac{D\rho_m}{Dt}&=\frac{1}{\eps} \Delta_s\rho_m.
\end{aligned}
\eeq
The third is the incompressible gradient flow associated to $\cF_2,$
\beq\label{e:cF2_IGF}
\mrV=\cI \left(-\mrG \kappa +\frac{\kappa^3}{2}\right).
\eeq

 To compare the evolution of curves  under three models we take identical initial data $U_0$ and supplement this with spatially uniform membrane density $\rho_m\equiv1$ for the flow \eqref{e:cF2_GF}. For these simple energies the final equilibrium of all three flows is a circle with a length equal, to leading order, to that of the initial curve.  It is instructive to compare the transients. The initial curve is a slightly asymmetric Trillium shape.  The end state of the three simulations at time $t=0.01$ is presented in Figure\,\ref{f:Incomp} (left). All three flows capture  the general tendency of the Willmore flow to drive evolution away from the initial non-convex shape to a convex one. The most significant difference is the rate of convergence to circular equilibrium. The constraint imposed by the mass motion, induces a lag in the shrinking of the curve length in the non-convex regions (green line), leading it to oscillate about the curve produced by the global constraint flow (black curve). These two curves have 6 roughly equally spaced crossings, with the black curve closer to circularity. This suggests a lower damping rate of the Fourier eigenmodes of the curve under the incompressible flow.   The finite $\eps$ penalty flow yields a curve that has compressed its length by $O(\eps)$.  The oscillations in the density $\rho_m$ induces peaks in the concave regions where the curve length  shortens. The asymmetry in the initial curve manifests itself in the different peak sizes, and an order of $\eps$ increase in average membrane density $\rho_m$ associated to leading-order preservation of curve length.

The numerical code is adapted from that described in \cite{NPW_25} and available on the associated GitHub page \cite{WGHub_25}.

\section{Discussion}

There are several applications to bio-membrane dynamics to which the analysis presented here can apply. A crucial application is to incorporate the incompressibility condition to gradient flows of interfacial energies that incorporate membrane adhesion and repulsion energies. Such effects are almost universal in biological membranes with embedded charges, see \cite{NPW_25} for a presentation of these. The electrostatic interactions yield local attraction while waters of solvation present a barrier to membrane fusion. The adhesion induces considerable local membrane stretching, incorporating an incompressibility prefactor could have a significant impact on the evolution.

A second application it to membranes formed from blends of components. The membrane incompressibility interacts non-trivially with packing and steric effects. Some components of membranes, particularly cholesterol, serve as interstitial agents that increase density without significantly contributing to volume. This suggests extending the membrane density to a multicomponent model with constituents $\rho:=(\rho_{1},\ldots, \rho_{N})$ that that form a blended membrane with a 
nonlinear packing energy. The incompressible penalty can take the form
$$ \cF_{\eps}(U,\rho)=\int_\mbbS \frac {( \rho^t \mrP \rho  -1)^2}{2\eps}\, \mrd \sigma,$$
where the packing compatibility matrix $\mrP\in\mbbR^{N\times N}$ is a fixed symmetric matrix with positive eigenvalues. The larger eigenvalues corresponding to eigenvectors whose constituent blend form an unfavorable (less dense) packing, while smaller eigenvalues represent favorable (more dense) packings. In the limit $\eps\to0^+$ the packings are constrained to reside on the $N-1$ dimensional packing ellipse
$$\Sigma^{N-1}(\mrP):=\left\{\rho\in\mbbR^N_+\,\bigl|\, \rho^t\mrP\rho=1\right\}.$$ See \cite{CDPV-20} for discussion of models of glycolipids and their role in membrane stability.

\section*{acknowledgment}
The first author recognizes support from the NSF through grant NSF-2205553. The third author recognizes support from NSERC Canada. The authors anticipate development of a freeware code library for general interfacial gradient flows described in this paper.

\bibliographystyle{siam}

\end{document}

%% file: Cal_mbb_mr.tex
\def\cA{{\mathcal A}}

\def\cC{{\mathcal C}}

\def\cF{{\mathcal F}}
\def\cG{{\mathcal G}}
\def\cH{{\mathcal H}}
\def\cI{{\mathcal I}}
\def\cJ{{\mathcal J}}

\def\cM{{\mathcal M}}
\def\cN{{\mathcal N}}

\def\cR{{\mathcal R}}

\def\cT{{\mathcal T}}

\def\mbbF{\mathbb{F}}

\def\mbbN{\mathbb{N}}

\def\mbbR{\mathbb{R}}
\def\mbbS{\mathbb{S}}

\def\mbbV{\mathbb{V}}

\def\mbbZ{\mathbb{Z}}

\def\mrF{\textrm F}
\def\mrG{\textrm G}
\def\mrH{\textrm H}
\def\mrI{\textrm I}
\def\mrJ{\textrm J}

\def\mrM{\textrm M}

\def\mrP{\textrm P}

\def\mrR{\textrm R}

\def\mrT{\textrm T}
\def\mrU{\textrm U}
\def\mrV{\textrm V}
\def\mrW{\textrm W}

\def\mrd{\textrm d}